\documentclass[letterpaper, 10pt]{article}
\usepackage[psamsfonts]{amssymb}
\usepackage{amsmath, amsthm, anysize, enumerate, color, graphicx}

\newtheorem{theorem}{Theorem}[section]
\newtheorem{question}[theorem]{Question}
\newtheorem{problem}[theorem]{Problem}

\newtheorem{lemma}[theorem]{Lemma}
\newtheorem{proposition}[theorem]{Proposition}
\newenvironment{proof_of}[1]{\noindent {\bf Proof of #1:}
	\hspace*{1mm}}{\hspace*{\fill} $\qedsymbol$ }
\newcommand{\widgraph}[2]{\includegraphics[keepaspectratio,width=#1]{#2}}

\theoremstyle{definition}
\newtheorem{remark}[theorem]{Remark}
\newtheorem{example}[theorem]{Example}

\newcommand{\E}{\mathbb{E}}
\newcommand{\R}{\mathbb{R}}
\newcommand{\N}{\mathbb{N}}
\newcommand{\Q}{\mathbb{Q}}
\renewcommand{\P}{\mathbb{P}}
\newcommand{\M}{\mathcal{M}}
\newcommand{\W}{\mathcal{W}}
\newcommand{\sgn}{\text{sgn}}
\newcommand{\conv}{\text{conv}}
\newcommand{\Dom}{\text{Dom}}
\newcommand{\stirlingtwo}[2]{\genfrac{\lbrace}{\rbrace}{0pt}{}{#1}{#2}}

\begin{document}

\title{Maximum entropy distributions on graphs}
\author{Christopher~J.~Hillar\thanks{Redwood Center for Theoretical
    Neuroscience, \texttt{chillar@msri.org}; partially supported by National Science Foundation Grant IIS-0917342 and 
    a National Science Foundation All-Institutes Postdoctoral Fellowship administered by the
    Mathematical Sciences Research Institute through its core grant
    DMS-0441170.} \hspace{20mm} Andre Wibisono\thanks{Department of Electrical Engineering and Computer
    Science, \texttt{aywibisono@wisc.edu}.} \\ $ $\\University of California, Berkeley}



\maketitle

\begin{abstract}
Inspired by applications to theories of coding and communication in networks of nervous tissue, we study maximum entropy distributions on weighted graphs with a given expected degree sequence. These distributions are characterized by independent edge weights parameterized by a shared vector of vertex potentials. Using the general theory of exponential family distributions, we derive the existence and uniqueness of the maximum likelihood estimator (MLE) of the vertex parameters. We also prove consistency of the MLE from a single  sample in the limit of large graphs, extending results of Chatterjee, Diaconis, and Sly in the unweighted case (the ``$\beta$-model" in statistics). Interestingly, our proofs require tight estimates on the norms of inverses of symmetric, diagonally dominant positive matrices.  Along the way, we derive analogues of the Erd\H{o}s-Gallai criterion of graphical degree sequences for weighted graphs.
\end{abstract}

\section{Introduction}
\label{Sec:Intro}

In this work, we explore some of the mathematics underlying three probabilistic models of
undirected weighted graphs with fixed sufficient statistics.  More specifically, we characterize and study the maximum entropy
graph distribution given a fixed expected degree sequence in each of the following three cases:
\begin{enumerate}[1.)]
  \item Graphs with bounded edge weights in $\{0,1,\ldots, r-1\}$, for a fixed integer $r \geq 2$;
  \item Graphs with weights in the nonnegative integers; and
  \item Graphs with continuous weights in $[0,\infty)$.
\end{enumerate}

Our main result here, extending findings of \cite{Chatterjee, rinaldo2013maximum} to the first case, is that a single graph 
sample from one of these distributions determines its defining parameters in the limit as the number of vertices in the network goes to infinity. 
Of course, translating such a verbal description into precise mathematics requires preparation, which we carry out in Sections \ref{Sec:General} and \ref{Sec:Specific} (e.g., see Theorems \ref{Thm:ConsistencyFiniteDisc}, \ref{Thm:ConsistencyInfiniteDisc}, and \ref{Thm:ConsistencyCont} concerning the first, second, and third case, respectively).  In the remainder of the introduction, we explain our motivation for studying these models as computational tools for biology as well as outline some of the main ideas.

The 20th century concept of ``maximum entropy" is a powerful normative principle in statistical modeling that gives mathematical precision to William of Ockham's ``razor" from seven hundred years ago.  Given no more knowledge of a system other than poolings of 
measured quantities (correlations, degrees, etc.), what is the most generic or simple model for the observed phenomena?  
One principled answer is to choose among all distributions satisfying the measured constraints that (unique) model with the largest Shannon entropy \cite{shannon48}, \cite[Chapter~12]{coverthomas}.  Notably, this approach was applied in statistical physics with great success by Jaynes \cite{jaynes1957}, who pioneered the concept.  More recently, experimentalists and computational biologists have started to reveal insights from data by embracing the application of this normative approach to their disciplines.  

For instance, maximum entropy distributions have been found to well-model neural population activity in retina \cite{Schneidman2006,shlens2006}, amino acid interactions in proteins \cite{seno2008maximum}, antibody diversity \cite{Mora2010}, and flock behavior \cite{Bialek2012}, among many other examples \cite{de2018}. More specifically to our context of graph distributions, they have been useful as network models of data both in discrete \cite{giusti2015clique} and weighted \cite{desmarais2012plos} settings. In all these cases, and especially where data is limited, results implying that few samples suffice to pin down an experimental finding are essential for scientific reliability.  Theoretical study of these statistical ensembles is also important for other practical matters such as improving algorithms for parameter estimation or determining model fit.

That being said, our primary motivation for embarking on this work was instead to promote the idea that mathematical properties of probabilistic models on graphs can be exploited for computation, perhaps even by biological organisms -- and, more specifically, nervous tissue -- in their mysteriously complex processing of matter.  The theory behind brains is in its infancy, but there are still several experimental constraints and design principles \cite{rieke1999, sterling2015principles} to guide the mathematical biologist. Brains appear to have discrete aspects (spikes, vesicles, etc.) as well as continuous ones (local field potentials, spike timing, etc.).  Additionally, some form of connectivity between functional units is ubiquitous, and there are several projects devoted to mapping the interacting structures (e.g., www.humanconnectomeproject.org). Nervous tissue also often seems to operate in a stochastic manner but nonetheless can have reliable features correlating with stimuli or behavior, such as average numbers of action potentials (i.e., spikes) in a window of time (so-called ``firing rates").
Another fundamental aspect of brain dynamics is the extraordinary parallelism inherent to its functioning, allowing for realtime complex computation with the energy requirements of an ordinary light bulb.\footnote{IBM's Watson Jeopardy champion required a factor of at least 4000 times more power than its human competitors.}
\begin{figure}[t!]
	\begin{center}
\includegraphics[width=1.5 in]{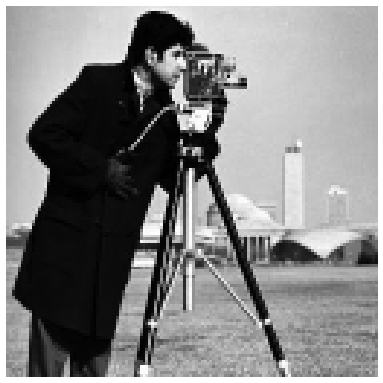}
\includegraphics[width=1.5 in]{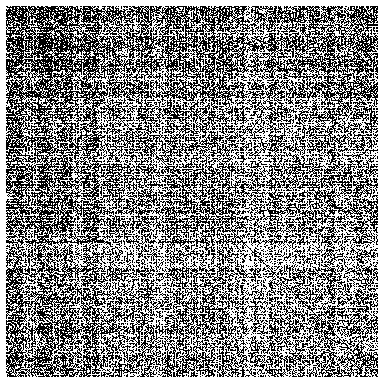}
\includegraphics[width=1.5 in]{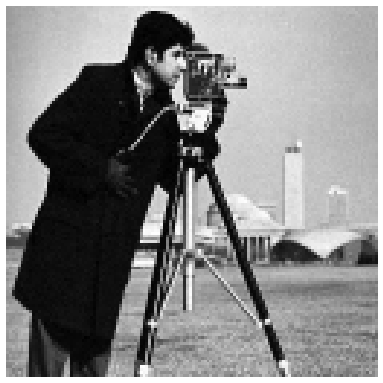}
\includegraphics[width=1.8 in]{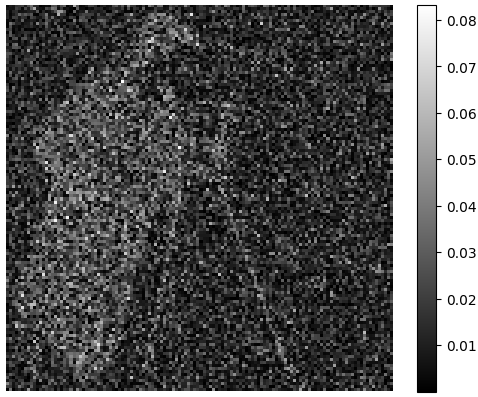}
\vspace{.15 cm}
\includegraphics[width=1.5 in]{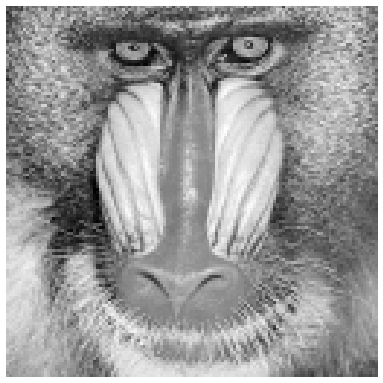}
\includegraphics[width=1.5 in]{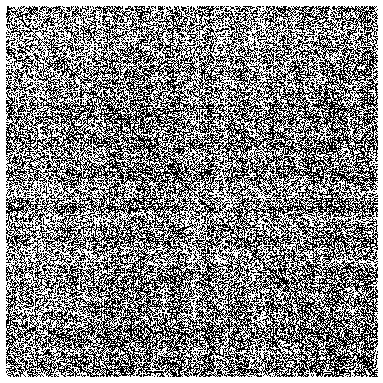}
\includegraphics[width=1.5 in]{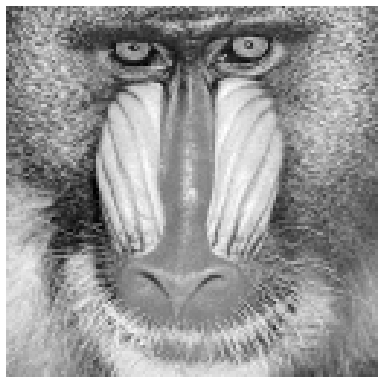}
\includegraphics[width=1.8 in]{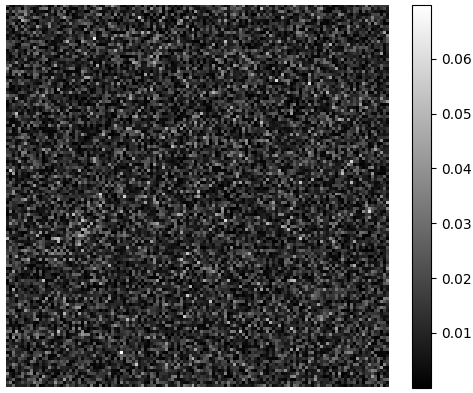}
\caption{\small{\textbf{Maximum entropy graph distributions for image quantization}. (Top row) From left to right:  $128 \times 128$ grayscale ``cameraman" image with mean zero and intensity values scaled to the interval $[-1, 1]$; two binary graph samples from a maximum entropy distribution using pixel intensity values as potentials (upper and lower triangular portions contain two adjacency matrices with white signifying edge presence); recovery of potentials from one graph sample using 10 iterations of the MLE inference algorithm from \cite{Chatterjee} (see, more generally, Theorem~\ref{Thm:MLEAlgFiniteDisc}); and absolute difference of recovery from original.  (Bottom row): Same as top row but for $128 \times 128$ ``baboon" image.}}
\label{fig_imgs_app_1}
\end{center}
\end{figure}

To summarize, biological constraints suggest models for information processing that are cooperative, stochastic, reliable, parallel, and cheap.
Thus, our motivating observation here is that coding with maximum entropy models on graphs -- in both discrete and continuous settings -- exhibits all of these properties.

Consider first the basic problem facing an organism of how to internally represent a continuous signal.  For instance, image capture in human retina occurs when rods (and color-sensitive cones in bright light) detect photons through the mechanical ``bending" of opsin molecules.  In particular, a $100\times100$ micrometer square patch of retina recruits 10,000 rods containing $10^{12}$ total opsins to detect a \textit{single photon} in starlight every 200 milliseconds.  In turn, the measurements taken inside these large photoreceptor arrays are converted into binary electrical pulses passed from retinal ganglion cells to the visual cortex through the optic nerve. An apparent continuum of stimuli are therefore converted by nervous tissue into sequences of spike times.

If we conceive of an action potential as a binary variable between pairs of connected neurons, the work of \cite{Chatterjee} suggests a strikingly simple stochastic coding scheme to solve this problem:  connect up a network with shared continuous values at each vertex and sample a simple graph in parallel (spatially) using biased ``coin flips" on edges.  For large networks, the original real values can be recovered with high fidelity from a measured degree sequence (with high probability); see Figure~\ref{fig_imgs_app_1}.  This common property of maximum entropy models that small numbers of samples determine them has been exploited before (e.g., for image processing \cite{geman1986markov}), but here we use this fact to stochastically, yet reliably, transform continuous into binary.

More precisely, let $\mathbf{x} = (x_1,\ldots,x_n)$ be a positive vector of parameters, which perhaps consists of $n$ continuous sensor measurements.
Provided such an $\mathbf{x} \in \mathbb R^n$, one can construct a distribution on binary graphs by assigning independent Bernoulli variables $A_{ij}$ on each edge $(i \neq j)$ with probability of presence given by $\frac{1}{x_i x_j + 1}$.  It is classical that given its expected degree sequence $\mathbf{d} = (d_1,\ldots,d_n)$:
\[ d_i = \sum_{j \neq i} \frac{1}{x_i x_j + 1}, \quad \text{ for } i = 1,\dots,n,\]
this distribution on graphs is the one with maximum entropy. 
In particular, when $\mathbf{x}$ is all $1$s, the corresponding maximum entropy graph distribution is Erd\H{o}s-R\'enyi given the expected degree constraints $d_i = (n - 1)/2$.
The situation is similar if integral values $r > 2$ on edges are allowed. 
Remarkably, a single such graph sample typically contains most of the information about the original parameters \cite{Chatterjee, rinaldo2013maximum}.  In fact, as we shall see, a maximum likelihood 
estimate (MLE) $\hat{\mathbf{x}}$ is very close to the original; moreover, it can be computed by solving the $n$ algebraic equations above for indeterminates $\mathbf{x}$ given the sample degree sequence $\mathbf{d}$.

To demonstrate these ideas visually, we present a toy example in Figure~\ref{fig_imgs_app_1} in which a $128 \times 128$ grayscale picture with pixel intensities scaled between -1 and 1 is 
transformed into binary.  In this case, each pixel value $\theta_i$ represents a vertex potential via $x_i = e^{\theta_i}$, and a binary graph sample is produced as described above.  To recover the 
original signal, rather than  solving the algebraic equations  directly, we use the fixed-point algorithm of \cite{Chatterjee} (Theorem~\ref{Thm:MLEAlgFiniteDisc} generalizes this to $r > 2$).  As
expected by theory, the recovery is a close match.

Taking this exploration further, suppose we do not fix \textit{a priori} bounds on the possible multiplicities of edges in a graph sample.  To again provide biological motivation, a common method of information transport in brains is via synaptic release of vesicles, each of which can contain thousands of neurotransmitter molecules (see Figure~\ref{fig_imgs_app_2_3}\textbf{a}).  Allowing any integer weights $k = 0, 1, 2, \ldots$ in this scenario, the maximum entropy distribution is again determined by potentials $\mathbf{x}$, but with the 
probability of a random edge value growing geometrically as $\mathbb P(A_{ij} = k) = \left(1 - \frac{1}{x_i x_j} \right)\frac{1}{(x_i x_j)^k}$.  We also have the following equations for expected degrees:
\[ d_i = \sum_{j \neq i} \frac{1}{x_i x_j - 1}, \quad \text{ for } i = 1,\dots,n.\]
As in the previous case, a single sample from this model can determine it with high precision (Theorem~\ref{Thm:ConsistencyInfiniteDisc}).

One of the hidden challenges in this work is determining when a given degree sequence can arise from a graph.  For simple graphs as in our first example, the corresponding characterization is called 
the \textit{Erd\H{o}s-Gallai criterion for graphical sequences} (see Theorem~\ref{Thm:GraphicalFiniteDisc}, which generalizes to all $r\geq2$). In the case of unbounded integer weights, we present the 
following result, which we have not been able to find in the literature.  

\begin{theorem}\label{Thm:GraphicalInfiniteDisc}
A sequence $(d_1, \dots, d_n)$ arises as the degrees from a graph with nonnegative integer weights if and only if $\sum_{i=1}^n d_i$ is even and: 
\begin{equation}\label{Eq:GraphicalInfiniteDisc}
\max_{1 \leq i \leq n} d_i \leq \frac{1}{2} \sum_{i=1}^n d_i.
\end{equation}
\end{theorem}

Our third and final maximum entropy scenario involves graph distributions with real-valued weights.  In neuroscience, this situation might apply to the case of gap junctions, which transmit continuous 
signals via direct ion exchange.  Another example might be the timing of action potentials, as a number of scientific articles have appeared suggesting that precise spike timing \cite{butts2007} and synchrony \cite{uhlhaas2009} 
are important for various computations in the brain.\footnote{We should note that it is well-known that precise spike timing is used for time-disparity computation in animals \cite{carr1993}, such as when owls track prey with binocular hearing or when electric fish use electric fields around their bodies for locating objects.}  
As we shall see in Section~\ref{Sec:Cont}, the analogous maximum entropy distribution is again parameterized by a real vector $\mathbf{x}$, and finding the maximum likelihood estimate for these parameters using a degree sequence $\mathbf{d}$ boils down to the following set of $n$ equations in the $n$ indeterminates $x_1,\ldots,x_n$:
\[d_i = \sum_{j \neq i} \frac{1}{x_i+x_j}, \quad \text{ for } i = 1,\dots,n.\]

Here, the edge weights $A_{ij}$ of the corresponding maximum entropy distribution end up being independent exponential random variables with mean $\frac{1}{x_i+x_j}$.
Intriguingly, this system of equations is also interesting from the perspective of (real) algebraic geometry and has been studied in a more general context using matroid theory 
\cite{SturmEntDisc} (see Remark~\ref{SturmfelsRemark} below for more details).  In particular, with tools from that discipline, it is possible to compute the exact real numbers representing the MLE (and their algebraic degrees), 
including when a fixed underlying connectivity graph is specified.  For instance, having a unique positive solution to the MLE can be verified using tools from algebraic geometry.
On the other hand, as far as we know, the geometric properties of the first two sets of equations above have not been studied.

\begin{figure}
	\begin{center}
\textbf{a)}\includegraphics[width=2.1 in]{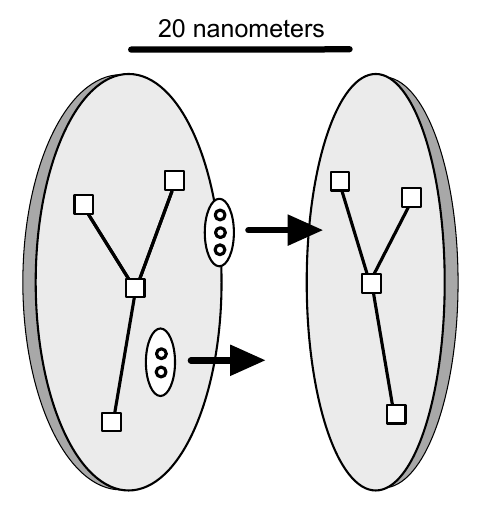}
\textbf{b)}\includegraphics[width=4.1 in]{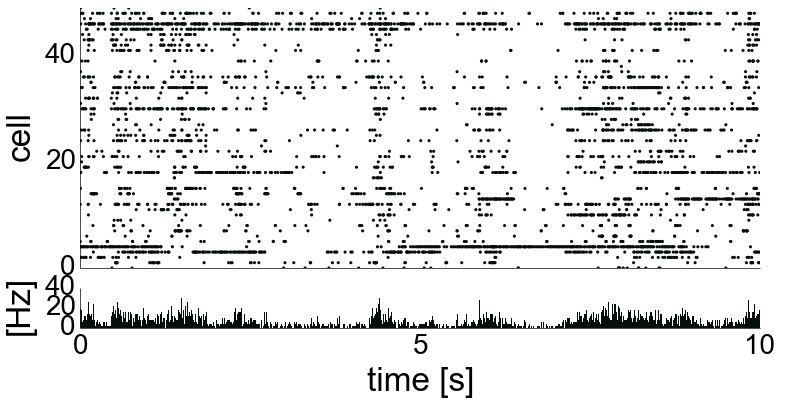}
\caption{\small{\textbf{Neural communication and timing}. \textbf{a}) Cartoon illustration of hypothetical vesicle release geometry at neural synapse. \textbf{b}) Spike timing has stochastic characteristics although reoccurring patterns in network dynamics can be uncovered using maximum entropy models applied to spontaneous \cite{hillar2015robust} and trial-based \cite{effenberger2015discovery} neural data.}}
\label{fig_imgs_app_2_3}
\end{center}
\end{figure}

We now make some brief remarks on our proofs of limiting consistency for the MLE.  The crucial object of interest is the map $\mathbf{x} \mapsto \textbf{d}$, and in our analysis of the previous two graph 
distributions, we require tight estimates of norms of inverses of certain Hessian matrices involving this map (see 
Sections~\ref{Sec:ProofConsistencyInfiniteDisc} and \ref{Sec:ProofConsistencyCont}, which use the matrix inequality from \cite{HLW} stated as Theorem \ref{Thm:Main} below).
Already combined with classical ``large deviation" results in bounded discrete variables (e.g., Bernstein's inequality \cite{bernstein1924}), such estimates 
suffice for the first model above.  For the remaining two models, we require a powerful generalization (Theorem~\ref{Thm:ConcIneqSubExp}) to unbounded random variables that has been 
fundamental for applications, such as compressed sensing \cite{vershynin}.

Of course, the universe never provides lunch  for free, and it must be emphasized that a single graph realization consists of many samples at edges, thereby explaining  statistical power.  Moreover, as a practical matter, the schemes discussed here are proofs-of-concept and not readily converted into engineering applications.  For instance, in the examples shown in Figure~\ref{fig_imgs_app_1}, although the original $8$-bit images have been projected into a binary space, its dimension is square the original and thus not efficient from a compression standpoint (although see Figure \ref{fig_imgs_app_sparse} for a sparse graph example).  On the other hand, these graph distributions do have some immediate practical implications for machine learning theory.  In particular, we have produced a model with no requirement for largest discrete sampled quantity.  We have not explored this connection, but there should be some low hanging fruit.  For more on further explorations and open questions, see Section~\ref{Sec:Discussion}.

Since the submission of this work, we have been made aware of progress by several other researchers on the models
appearing here.  In particular, Rinaldo, Petrovi{\'c}, and Fienberg \cite{rinaldo2013maximum} prove several results in the case of bounded integer weights (the ``generalized $\beta$-model"), as well as theorems for many other ensembles (e.g., the Rasch, Bradley-Terry, and $p_1$ models). For some additional asymptotic results concerning these maximum entropy distributions, see \cite{yan2015asymptotic, yan2016asymptotics};  recent extensions of the $\beta$-model in networks incorporating covariate information can be found in \cite{wahlstrom2017beta, yan2018statistical}.
Although this article was inspired by biology, we should mention that the econometrics community \cite{graham2017econometric} is also interested in how coarse measurements from agents can determine their shared latent properties reliably as network size grows. For a recent general overview of random networks (e.g., inhomogeneous random graphs \cite{bollobas2007phase}), we refer to the book \cite{van2016random}.

The organization of this paper is as follows.  In Section~\ref{Sec:General}, we lay out the general theory of maximum entropy distributions on weighted graphs. In Section~\ref{Sec:Specific}, we specialize the general theory to each of our three classes of weighted graphs, provide an explicit characterization of the maximum entropy distributions, and prove a generalization of the Erd\H{o}s-Gallai criterion for weighted graphical sequences. Furthermore, we also state a limiting consistency property of the MLE of the vertex parameters from one graph sample (Theorems~\ref{Thm:ConsistencyFiniteDisc}, \ref{Thm:ConsistencyInfiniteDisc}, and \ref{Thm:ConsistencyCont}) and present an efficient fixed-point algorithm in the bounded discrete case (Theorem~\ref{Thm:MLEAlgFiniteDisc}). Section~\ref{Sec:Proofs} then provides proofs of the main technical results presented in Section~\ref{Sec:Specific}. Finally, Section~\ref{Sec:Discussion} closes with a discussion that includes some future research directions.  As our audience contains mathematical biologists, we have tried to keep the mathematics relatively self-contained.

\paragraph{Notation.}
Let $\R_+ = (0, \infty)$, $\R_0 = [0,\infty)$, $\N = \{1,2,\dots\}$, and $\N_0 = \{0,1,2,\dots\}$. The natural logarithm is denoted by $\log$.  We write $\sum_{\{i,j\}}$ and $\prod_{\{i,j\}}$ for the summation and product, respectively, over all $\binom{n}{2}$ pairs $\{i,j\}$ with $i \neq j$. For a subset $C \subseteq \R^n$, $C^\circ$ and $\overline{C}$ denote the interior and closure of $C$ in $\R^n$, respectively. For a vector $\mathbf{x} = (x_1,\dots,x_n) \in \R^n$, $\|\mathbf{x}\|_1 = \sum_{i=1}^n |x_i|$ and $\|\mathbf{x}\|_\infty = \max_{1 \leq i \leq n} |x_i|$ denote the $\ell_1$ and $\ell_\infty$ norms of $\mathbf{x}$. For an $n \times n$ matrix $J = (J_{ij})$, the matrix norm $\|J\|_\infty$ induced by the $\|\cdot\|_\infty$-norm on vectors in $\R^n$ is:
\begin{equation*}
\|J\|_\infty = \max_{\mathbf{x} \neq 0} \frac{\|J \mathbf{x}\|_\infty}{\|\mathbf{x}\|_\infty} = \max_{1 \leq i \leq n} \sum_{j=1}^n |J_{ij}|.
\end{equation*}

\section{General theory via exponential family distributions}
\label{Sec:General}

In this section we develop the general machinery of maximum entropy distributions on graphs via the theory of exponential family distributions~\cite{brown1986, barndorff2014information}, and in subsequent sections we specialize our analysis to some particular cases of weighted graphs.  Our treatment of basic theory 
mainly follows \cite{Jordan}, which is geared toward practitioners of machine learning.

Consider an undirected graph $G$ on $n \geq 3$ vertices with edge $(i,j)$ having weight $a_{ij} \in S$, where $S \subseteq \R$ is the set of possible weight values. We consider the following specific cases: 1.) Finite discrete weighted graphs with edge weights in $S = \{0,1,\dots,r-1\}$, 2.) Infinite discrete weighted graphs with edge weights in $S = \N_0$, and 3.) Continuous weighted graphs with edge weights in $S = \R_0$.  A graph $G$ is fully specified by its \textit{adjacency matrix} $\mathbf{a} = (a_{ij})_{i,j=1}^n$, which is an $n \times n$ symmetric matrix with zeros along its diagonal. For fixed $n$, a probability distribution over graphs $G$ corresponds to a distribution over adjacency matrices $\mathbf{a} = (a_{ij}) \in S^{\binom{n}{2}}$. Given a graph with adjacency matrix $\mathbf{a} = (a_{ij})$, let $\deg_i(\mathbf{a}) = \sum_{j \neq i} a_{ij}$ be the degree of vertex $i$, and let $\deg(\mathbf{a}) = (\deg_1(\mathbf{a}), \dots, \deg_n(\mathbf{a}))$ be the degree sequence of $\mathbf{a}$.

\subsection{Characterization of maximum entropy distribution}

Let $\mathcal{S}$ be a $\sigma$-algebra over the set of weight values $S$, and assume there is a canonical $\sigma$-finite probability measure $\nu$ on $(S,\mathcal{S})$. Let $\nu^{\binom{n}{2}}$ be the product measure on $S^{\binom{n}{2}}$, and let $\mathfrak{P}$ be the set of all probability distributions on $S^{\binom{n}{2}}$ that are absolutely continuous with respect to $\nu^{\binom{n}{2}}$. Since $\nu^{\binom{n}{2}}$ is $\sigma$-finite, these probability distributions can be characterized by their density functions, i.e.\ the Radon-Nikodym derivatives with respect to $\nu^{\binom{n}{2}}$. Given a sequence $\mathbf{d} = (d_1, \dots, d_n) \in \R^n$, let $\mathfrak{P}_\mathbf{d}$ be the set of distributions in $\mathfrak{P}$ whose expected degree sequence is equal to $\mathbf{d}$:
\begin{equation*}
\mathfrak{P}_\mathbf{d} = \{ \P \in \mathfrak{P} \colon \E_\P[\deg(A)] = \mathbf{d}\},
\end{equation*}
where in the definition above, the random variable $A = (A_{ij}) \in S^{\binom{n}{2}}$ is drawn from the distribution $\P$. Then the distribution $\P^\ast$ in $\mathfrak{P}_\mathbf{d}$ with maximum entropy is precisely the exponential family distribution with the degree sequence as sufficient statistics~\cite[Chapter~3]{Jordan}. Specifically, the density of $\P^\ast$ at $\mathbf{a} = (a_{ij}) \in S^{\binom{n}{2}}$ is given by:\footnote{We choose to use $-\theta$ in the parameterization~\eqref{Eq:MaxEntDist}, instead of the canonical parameterization $p^\ast(\mathbf{a}) \propto \exp(\theta^\top \deg(\mathbf{a}))$, because it simplifies the notation in our later presentation.}
\begin{equation}\label{Eq:MaxEntDist}
p^\ast(\mathbf{a}) = \exp \big( -\theta^\top \deg(\mathbf{a}) - Z(\theta) \big),
\end{equation}
where $Z(\theta)$ is the \emph{log-partition function}:
\begin{equation*}
Z(\theta) = \log \int_{S^{\binom{n}{2}}} \exp\big( -\theta^\top \deg(\mathbf{a}) \big) \; \nu^{\binom{n}{2}}(d\mathbf{a}),
\end{equation*}
and $\theta = (\theta_1, \dots, \theta_n)$ is a parameter that belongs to the \emph{natural parameter space}:
\begin{equation*}
\Theta = \{\mathbf{\theta} \in \R^n \colon Z(\mathbf{\theta}) < \infty\}.
\end{equation*}
We will also write $\P^\ast_\theta$ if we need to emphasize the dependence of $\P^\ast$ on the parameter $\theta$.

Using the definition $\deg_i(\mathbf{a}) = \sum_{j \neq i} a_{ij}$, we can write:
\begin{equation*}
\exp \big( -\theta^\top \deg(\mathbf{a}) \big)
= \exp \Big( -\sum_{\{i,j\}} (\theta_i+\theta_j) a_{ij} \Big)
= \prod_{\{i,j\}} \exp \big(-(\theta_i+\theta_j) a_{ij} \big).
\end{equation*}
Hence, we can express the log-partition function as:
\begin{equation}\label{Eq:DecompositionLogPartition}
Z(\theta) = \log \prod_{\{i,j\}} \int_S \exp \big(-(\theta_i+\theta_j) a_{ij} \big) \; \nu(da_{ij}) = \sum_{\{i,j\}} Z_1(\theta_i+\theta_j),
\end{equation}
in which $Z_1(t)$ is the marginal log-partition function:
\begin{equation*}
Z_1(t) = \log \int_S \exp (-ta) \: \nu(da).
\end{equation*}
Consequently, the density in~\eqref{Eq:MaxEntDist} can be written as:
\begin{equation*}
p^\ast(\mathbf{a}) = \prod_{\{i,j\}} \exp \big( -(\theta_i + \theta_j) a_{ij} - Z_1(\theta_i + \theta_j) \big).
\end{equation*}
This means the edge weights $A_{ij}$ are independent random variables, with $A_{ij} \in S$ having distribution $\P_{ij}^\ast$ with density:
\begin{equation*}
p_{ij}^\ast(a) = \exp \big( -(\theta_i + \theta_j) a - Z_1(\theta_i + \theta_j) \big).
\end{equation*}
In particular, the edge weights $A_{ij}$ belong to the same exponential family distribution but with different parameters that depend on $\theta_i$ and $\theta_j$ (or rather, on their sum $\theta_i + \theta_j$). The parameters $\theta_1, \dots, \theta_n$ can be interpreted as the potential at each vertex that determines how strongly the vertices are connected to each other. Furthermore, we can write the natural parameter space $\Theta$ as:
\begin{equation*}
\Theta = \{\theta \in \R^n \colon Z_1(\theta_i+\theta_j) < \infty \; \text{ for all } \; i \neq j \}. 
\end{equation*}

\subsection{Maximum likelihood estimator and moment-matching equation}

Using the characterization of $\P^\ast$ as the maximum entropy distribution in $\mathfrak{P}_\mathbf{d}$, the condition $\P^\ast \in \mathfrak{P}_\mathbf{d}$ means we need to choose the parameter $\theta$ for $\P^\ast_\theta$ such that $\E_\theta[\deg(A)] = \mathbf{d}$.\footnote{Here we write $\E_\theta$ in place of $\E_{\P^\ast}$ to emphasize the dependence of the expectation on the parameter $\theta$.} This is an instance of the moment-matching equation, which, in the case of exponential family distributions, is well-known to be equivalent to finding the maximum likelihood estimator (MLE) of $\theta$ given an empirical degree sequence $\mathbf{d} \in \R^n$.

Specifically, suppose we draw graph samples $G_1,\dots,G_m$ i.i.d.\ from the distribution $\P^\ast$ with parameter $\theta^\ast$, and we want to find the MLE $\hat \theta$ of $\theta^\ast$ based on the observations $G_1,\dots,G_m$. Using the parametric form of the density~\eqref{Eq:MaxEntDist}, this is equivalent to solving the following maximization problem:
\begin{equation*}
\max_{\theta \in \Theta} \: \mathcal{F}(\theta) \equiv -\theta^\top \mathbf{d} - Z(\theta),
\end{equation*}
where $\mathbf{d}$ is the average of the degree sequences of $G_1,\dots,G_m$. Setting the gradient of $\mathcal{F}(\theta)$ to zero reveals that the MLE $\hat \theta$ satisfies:
\begin{equation}\label{Eq:MLEEquation}
-\nabla Z(\hat \theta) = \mathbf{d}.
\end{equation}
Recall that the gradient of the log-partition function in an exponential family distribution is equal to the expected sufficient statistics. In our case, we have $-\nabla Z(\hat \theta) = \E_{\hat \theta}[\deg(A)]$, so the MLE equation~\eqref{Eq:MLEEquation} recovers the moment-matching equation $\E_{\hat \theta}[\deg(A)] = \mathbf{d}$.

In Section~\ref{Sec:Specific} we study the properties of the MLE of $\theta$ from a {\em single} sample $G \sim \P^\ast_\theta$. In the remainder of this section, we address the question of the existence and uniqueness of the MLE with a given empirical degree sequence $\mathbf{d}$.

Define the \emph{mean parameter space} $\M$ to be the set of expected degree sequences from all distributions on $S^{\binom{n}{2}}$ that are absolutely continuous with respect to $\nu^{\binom{n}{2}}$:
\begin{equation*}
\M = \{ \E_\P[\deg(A)] \colon \P \in \mathfrak{P} \}.
\end{equation*}
The set $\M$ is necessarily convex, since a convex combination of probability distributions in $\mathfrak{P}$ is also a probability distribution in $\mathfrak{P}$. Recall that an exponential family distribution is {\em minimal} if there is no linear combination of the sufficient statistics that is constant almost surely with respect to the base distribution. This minimality property clearly holds for $\P^\ast$, for which the sufficient statistics are the degree sequence. We say that $\P^\ast$ is \emph{regular} if the natural parameter space $\Theta$ is open. By the general theory of exponential family distributions~\cite[Theorem~3.3]{Jordan}, in a regular and minimal exponential family distribution, the gradient of the log-partition function maps the natural parameter space $\Theta$ to the interior of the mean parameter space $\M$, and this mapping:\footnote{The mapping is $-\nabla Z$, instead of $\nabla Z$, because of our choice of the parameterization in~\eqref{Eq:MaxEntDist} using $-\theta$.}
\begin{equation*}
-\nabla Z \colon \Theta \to \M^\circ
\end{equation*}
is bijective. We summarize the preceding discussion in the following result (see also \cite[Theorem 3.1]{rinaldo2013maximum} and surrounding remarks).

\begin{proposition}\label{Prop:RegularMinimal}
Assume $\Theta$ is open. Then there exists a solution $\theta \in \Theta$ to the MLE equation $\E_{\theta}[\deg(A)] = \mathbf{d}$ if and only if $\mathbf{d} \in \M^\circ$, and if such a solution exists then it is unique.
\end{proposition}

We now characterize the mean parameter space $\M$. We say that a sequence $\mathbf{d} = (d_1, \dots, d_n)$ is \emph{graphic} (or a {\em graphical sequence}) if $\mathbf{d}$ is the degree sequence of a graph $G$ with edge weights in $S$, and in this case we say that $G$ \emph{realizes} $\mathbf{d}$. It is important to note that whether a sequence $\mathbf{d}$ is graphic depends on the weight set $S$, which we consider fixed for now.

\begin{proposition}\label{Prop:MConvW}
Let $\W$ be the set of all graphical sequences, and let $\conv(\W)$ be the convex hull of $\W$. Then $\M \subseteq \conv(\W)$. Furthermore, if $\mathfrak{P}$ contains the Dirac delta measures, then $\M = \conv(\W)$.
\end{proposition}
\begin{proof}
The inclusion $\M \subseteq \conv(\W)$ is clear, since any element of $\M$ is of the form $\E_\P[\deg(A)]$ for some distribution $\P$ and $\deg(A) \in \W$ for every realization of the random variable $A$. Now suppose $\mathfrak{P}$ contains the Dirac delta measures $\delta_B$ for each $B \in S^{\binom{n}{2}}$. Given $\mathbf{d} \in \W$, let $B$ be the adjacency matrix of the graph that realizes $\mathbf{d}$. Then $\mathbf{d} = \E_{\delta_B}[\deg(A)] \in \M$, which means $\W \subseteq \M$, and hence $\conv(\W) \subseteq \M$ since $\M$ is convex.
\end{proof}

As we shall see in Section~\ref{Sec:Specific}, the result above allows us to conclude that $\M = \conv(\W)$ for the case of discrete weighted graphs. On the other hand, for the case of continuous weighted graphs we need to prove $\M = \conv(\W)$ directly since $\mathfrak{P}$ in this case does not contain the Dirac measures.

\begin{remark}
We emphasize the distinction between a \emph{valid} solution $\theta \in \Theta$ and a \emph{general} solution $\theta \in \R^n$ to the MLE equation $\E_\theta[\deg(A)] = \mathbf{d}$. As we saw from Proposition~\ref{Prop:RegularMinimal}, we have a precise characterization of the existence and uniqueness of the valid solution $\theta \in \Theta$, but in general, there are multiple solutions $\theta$ to the MLE equation. In this paper we shall be concerned only with the valid solution; Sanyal, Sturmfels, and Vinzant study some algebraic properties of general solutions~\cite{SturmEntDisc}.
\end{remark}

We close this section by discussing the symmetry of the valid solution to the MLE equation. Recall the decomposition~\eqref{Eq:DecompositionLogPartition} of the log-partition function $Z(\theta)$ into the marginal log-partition functions $Z_1(\theta_i+\theta_j)$. Let $\Dom(Z_1) = \{t \in \R \colon Z_1(t) < \infty\}$, and let $\mu \colon \Dom(Z_1) \to \R$ denote the (marginal) \emph{mean function}:
\begin{equation*}
\mu(t) = \int_S a \: \exp \big(-ta - Z_1(t) \big) \: \nu(da).
\end{equation*}
Observing that we can write:
\begin{equation*}
\E_\theta[A_{ij}] = \int_S a \: \exp \big( -(\theta_i + \theta_j) a - Z_1(\theta_i + \theta_j) \big) \: \nu(da) = \mu(\theta_i+\theta_j),
\end{equation*}
the MLE equation $\E_\theta[\deg(A)] = \mathbf{d}$ then becomes:
\begin{equation}\label{Eq:MLE-Sym}
d_i = \sum_{j \neq i} \mu(\theta_i+\theta_j), \quad \text{ for } i = 1, \dots, n.
\end{equation}

In the statement below, $\sgn$ denotes the sign function: $\sgn(t) = t/|t|$ if $t \neq 0$, and $\sgn(0) = 0$.

\begin{proposition}\label{Prop:SymSoln}
Let $\mathbf{d} \in \M^\circ$, and let $\theta \in \Theta$ be the unique solution to the system of equations~\eqref{Eq:MLE-Sym}. If $\mu$ is strictly increasing, then:
\begin{equation*}
\sgn(d_i-d_j) = \sgn(\theta_i-\theta_j), \quad \text{ for all } i \neq j,
\end{equation*}
and similarly, if $\mu$ is strictly decreasing, then:
\begin{equation*}
\sgn(d_i-d_j) = \sgn(\theta_j-\theta_i), \quad \text{ for all } i \neq j.
\end{equation*}
\end{proposition}
\begin{proof}
Given $i \neq j$, we have:
\begin{equation*}
\begin{split}
d_i-d_j
&= \Big( \mu(\theta_i+\theta_j) + \sum_{k \neq i,j} \mu(\theta_i+\theta_k) \Big)
- \Big( \mu(\theta_j+\theta_i) + \sum_{k \neq i,j} \mu(\theta_j+\theta_k) \Big) \\
&= \sum_{k \neq i,j} \big( \mu(\theta_i+\theta_k) - \mu(\theta_j+\theta_k) \big).
\end{split}
\end{equation*}
If $\mu$ is strictly increasing, then $\mu(\theta_i+\theta_k)-\mu(\theta_j+\theta_k)$ has the same sign as $\theta_i-\theta_j$ for each $k \neq i,j$, and thus $d_i-d_j$ also has the same sign as $\theta_i-\theta_j$. Similarly, if $\mu$ is strictly decreasing, then $\mu(\theta_i+\theta_k)-\mu(\theta_j+\theta_k)$ has the opposite sign of $\theta_i-\theta_j$, and thus $d_i-d_j$ also has the opposite sign of $\theta_i-\theta_j$.
\end{proof}

\section{Analysis for specific edge weights}
\label{Sec:Specific}

In this section we analyze the maximum entropy random graph distributions for several specific choices of the weight set $S$. For each case, we specify the distribution of the edge weights $A_{ij}$, the mean function $\mu$, the natural parameter space $\Theta$, and characterize the mean parameter space $\M$. We also study the problem of finding the MLE $\hat \theta$ of $\theta$ from one graph sample $G \sim \P^\ast_\theta$ and prove the existence, uniqueness, and consistency of the MLE. Along the way, we derive analogues of the Erd\H{o}s-Gallai criterion of graphical sequences for weighted graphs. We defer the proofs of the results presented here to Section~\ref{Sec:Proofs}.

\subsection{Finite discrete weighted graphs}
\label{Sec:FiniteDiscrete}

We first study weighted graphs with edge weights in the finite discrete set $S = \{0,1,\dots,r-1\}$, where $r \geq 2$. The case $r = 2$ corresponds to unweighted graphs, and our analysis in this section overlaps with some of that in \cite{Chatterjee, rinaldo2013maximum}. Proofs are outlined in Section~\ref{Sec:ProofFiniteDisc}.

\subsubsection{Characterization of the distribution}

We take $\nu$ to be the counting measure on $S$. Following the development in Section~\ref{Sec:General}, the edge weights $A_{ij} \in S$ are independent random variables with density:
\begin{equation*}
p_{ij}^\ast(a) = \exp\big(-(\theta_i+\theta_j)a - Z_1(\theta_i+\theta_j)\big), \quad 0 \leq a \leq r-1,
\end{equation*}
where the marginal log-partition function $Z_1$ is given by:
\begin{equation*}
Z_1(t) = \log \sum_{a = 0}^{r-1} \exp(-at) =
\begin{cases}
\log \frac{1-\exp(-rt)}{1-\exp(-t)} \quad & \text{ if } t \neq 0,\\
\log r &\text{ if } t = 0.
\end{cases}
\end{equation*}

Since $Z_1(t) < \infty$ for all $t \in \R$, the natural parameter space $\Theta = \{\theta \in \R^n \colon Z_1(\theta_i+\theta_j) < \infty, i \neq j\}$ is given by $\Theta = \R^n$. The mean function is given by:
\begin{equation}\label{Eq:MeanFuncFiniteDiscrete}
\mu(t) = \sum_{a=0}^{r-1} a \exp(-at - Z_1(t))
= \frac{\sum_{a=0}^{r-1} a \: \exp(-at)}{\sum_{a=0}^{r-1} \exp(-at)}.
\end{equation}
At $t = 0$ the mean function takes the value:
\begin{equation*}
\mu(0) = \frac{\sum_{a=0}^{r-1} a}{r} = \frac{r-1}{2},
\end{equation*}
while for $t \neq 0$, the mean function simplifies to:
\begin{equation}\label{Eq:MeanFuncFiniteDiscreteAlt}
\mu(t)
= -\left(\frac{1-\exp(-t)}{1-\exp(-rt)}\right) \cdot \frac{d}{dt} \sum_{a=0}^{r-1} \exp(-at)
= \frac{1}{\exp(t)-1}-\frac{r}{\exp(rt)-1}.
\end{equation}
Figure~\ref{Fig:FiniteDiscMean} shows the behavior of the mean function $\mu(t)$ and its derivative $\mu'(t)$ as $r$ varies.

\begin{figure}[h]
\begin{center}
\begin{tabular}{cc}
\widgraph{0.48\textwidth}{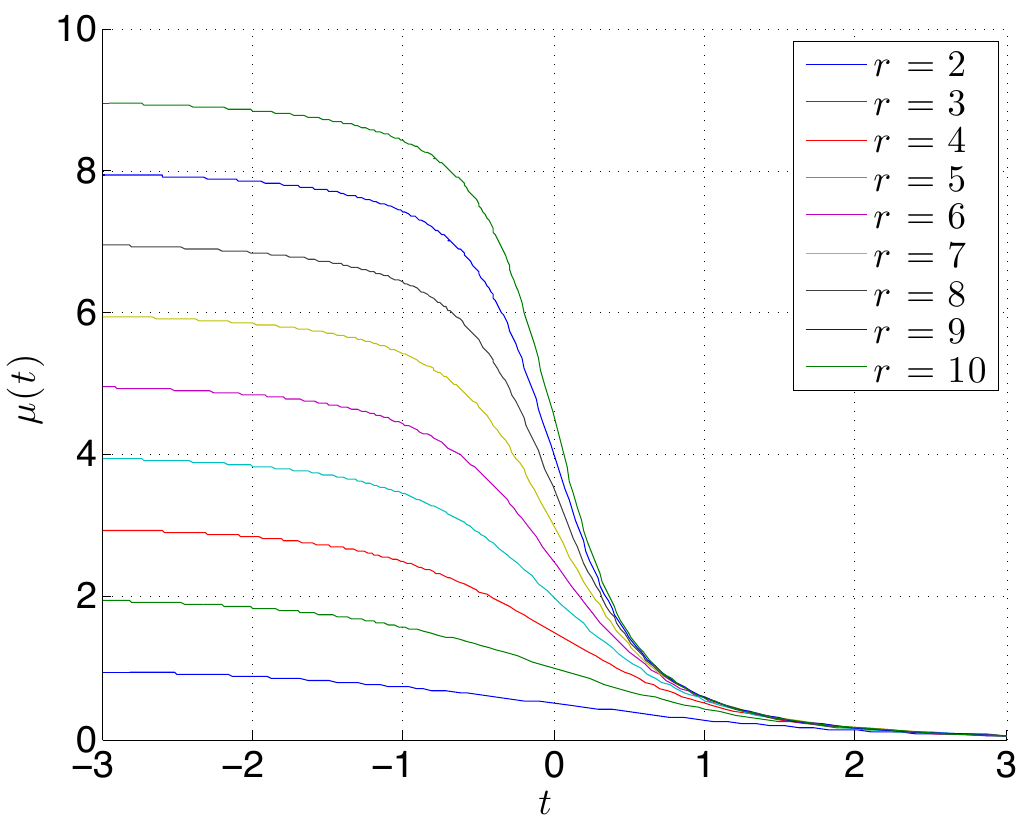}  &
\widgraph{0.48\textwidth}{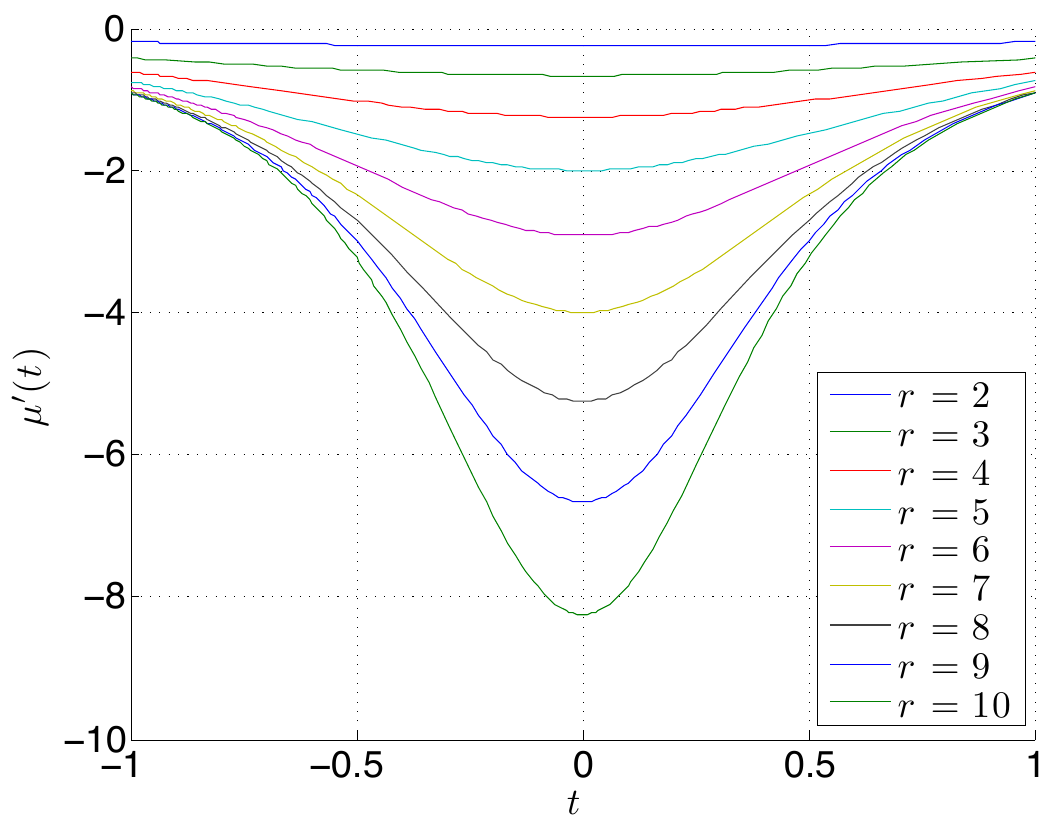} \\

(a) & (b)
\end{tabular}
\caption{Plot of the mean function $\mu(t)$ (left) and its derivative $\mu'(t)$ (right) as $r$ varies.}
\label{Fig:FiniteDiscMean}
\end{center}
\end{figure}

\begin{remark}
For $r = 2$, the edge weights $A_{ij}$ are independent Bernoulli random variables with:
\begin{equation*}
\P^\ast(A_{ij} = 1) = \mu(\theta_i+\theta_j) = \frac{\exp(-\theta_i-\theta_j)}{\exp(-\theta_i - \theta_j) + 1} = \frac{1}{\exp(\theta_i+\theta_j) + 1}.
\end{equation*}
As noted above, this is the model originally studied by Chatterjee, Diaconis, and Sly~\cite{Chatterjee} in the context of graph limits. When $\theta_1 = \theta_2 = \dots = \theta_n = t$, we recover the classical Erd\H{o}s-R\'enyi model with edge emission probability $p = 1/(1+\exp(2t))$.
\end{remark}

\subsubsection{Existence, uniqueness, and consistency of the MLE}

Consider the problem of finding the MLE of $\theta$ from one graph sample. Specifically, let $\theta \in \Theta$ and suppose we draw a sample $G \sim \P^\ast_\theta$. Then, as we saw in Section~\ref{Sec:General}, the MLE $\hat \theta$ of $\theta$ is a solution to the moment-matching equation $\E_{\hat \theta}[\deg(A)] = \mathbf{d}$, where $\mathbf{d}$ is the degree sequence of the sample graph $G$. As in~\eqref{Eq:MLE-Sym}, the moment-matching equation is equivalent to the following system:
\begin{equation}\label{Eq:MLEEqFiniteDisc}
d_i = \sum_{j \neq i} \mu(\hat \theta_i + \hat \theta_j), \quad i = 1,\dots,n.
\end{equation}

\begin{remark}
By setting $x_i = e^{\hat \theta_i}$ in~\eqref{Eq:MLEEqFiniteDisc}, we recover the expected degree equations from the introduction.
\end{remark}

Since the natural parameter space $\Theta = \R^n$ is open, Proposition~\ref{Prop:RegularMinimal} implies that the MLE $\hat \theta$ exists and is unique if and only if the empirical degree sequence $\mathbf{d}$ belongs to the interior $\M^\circ$ of the mean parameter space $\M$.

We also note that since $\nu^{\binom{n}{2}}$ is the counting measure on $S^{\binom{n}{2}}$, all  distributions on $S^{\binom{n}{2}}$ are absolutely continuous with respect to $\nu^{\binom{n}{2}}$, so $\mathfrak{P}$ contains all probability distributions on $S^{\binom{n}{2}}$. In particular, $\mathfrak{P}$ contains the Dirac measures, and by Proposition~\ref{Prop:MConvW}, this implies $\M = \conv(\W)$, where $\W$ is the set of all graphical sequences.

The following result characterizes when $\mathbf{d}$ is a degree sequence of a weighted graph with edge weights in $S$; we also refer to such $\mathbf{d}$ as a {\em (finite discrete) graphical sequence}. The case $r = 2$ recovers the classical Erd\H{o}s-Gallai criterion~\cite{ErdosGallai}.

\begin{theorem}\label{Thm:GraphicalFiniteDisc}
A sequence $(d_1,\dots,d_n) \in \N_0^n$ with $d_1 \geq d_2 \geq \dots \geq d_n$ is the degree sequence of a graph $G$ with edge weights in the set $S = \{0,1,\dots,r-1\}$, if and only if $\sum_{i=1}^n d_i$ is even and:
\begin{equation}\label{Eq:GraphicalFiniteDisc}
\sum_{i=1}^k d_i \leq (r-1)k(k-1) + \sum_{j=k+1}^n \min\{d_j, (r-1)k\}, \quad \text{ for } k = 1,\dots,n.
\end{equation}
\end{theorem}

Although the result above provides a precise characterization of the set of graphical sequences $\W$, it is not immediately clear how to characterize the convex hull $\conv(\W)$, or how to decide whether a given $\mathbf{d}$ belongs to $\M^\circ = \conv(\W)^\circ$. Fortunately, in practice we can circumvent this issue by employing the following algorithm to compute the MLE. The case $r = 2$ recovers the fixed-point algorithm proposed by Chatterjee et al.~\cite{Chatterjee} in the case of unweighted graphs.  For specific guarantees of the existence of the MLE, see \cite[Theorem~5.1]{rinaldo2013maximum}.

\begin{theorem}\label{Thm:MLEAlgFiniteDisc}
Given $\mathbf{d} = (d_1, \dots, d_n) \in \R_+^n$, define the function $\varphi \colon \R^n \to \R^n$ by $\varphi(\mathbf{x}) = (\varphi_1(\mathbf{x}), \dots, \varphi_n(\mathbf{x}))$, where
\begin{equation}\label{Eq:VarphiFiniteDisc}
\varphi_i(\mathbf{x}) = x_i + \frac{1}{r-1} \left( \log \sum_{j \neq i} \mu(x_i+x_j) -  \log d_i\right).
\end{equation}
Starting from any $\theta^{(0)} \in \R^n$, define
\begin{equation}\label{Eq:MLEAlgFiniteDisc}
\theta^{(k+1)} = \varphi(\theta^{(k)}), \quad k \in \N_0.
\end{equation}
Suppose $\mathbf{d} \in \conv(\W)^\circ$, so the MLE equation~\eqref{Eq:MLEEqFiniteDisc} has a unique solution $\hat \theta$. Then $\hat \theta$ is a fixed-point of the function $\varphi$, and the iterates~\eqref{Eq:MLEAlgFiniteDisc} converge to $\hat \theta$ geometrically fast: there exists a constant $\beta \in (0,1)$ that only depends on $(\|\hat \theta\|_\infty, \|\theta^{(0)}\|_\infty)$, such that:
\begin{equation}\label{Eq:MLERateOfConvFiniteDisc}
\|\theta^{(k)}-\hat\theta\|_\infty \leq \beta^{k-1} \: \|\theta^{(0)}-\hat\theta\|_\infty, \quad k \in \N_0.
\end{equation}
Conversely, if $\mathbf{d} \notin \conv(\W)^\circ$, then the sequence $\{\theta^{(k)}\}$ has a divergent subsequence.
\end{theorem}

\begin{figure}[h]
\begin{center}
\begin{tabular}{ccc}
\widgraph{0.35\textwidth}{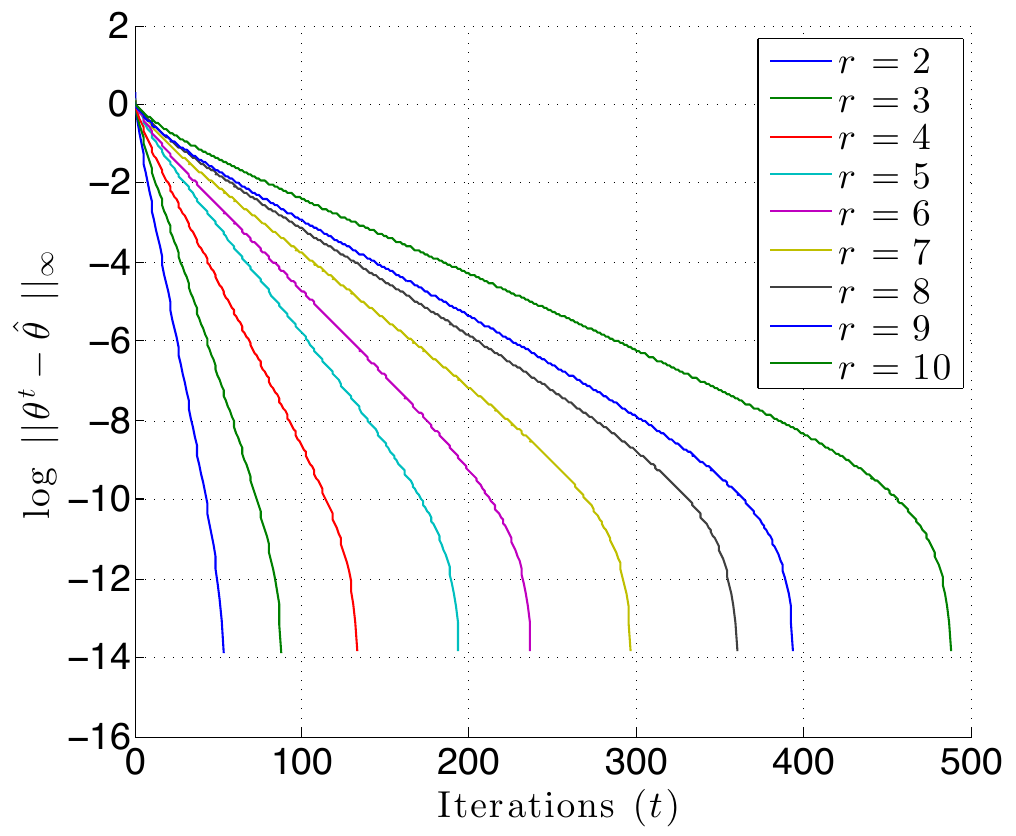}  &
\widgraph{0.29\textwidth}{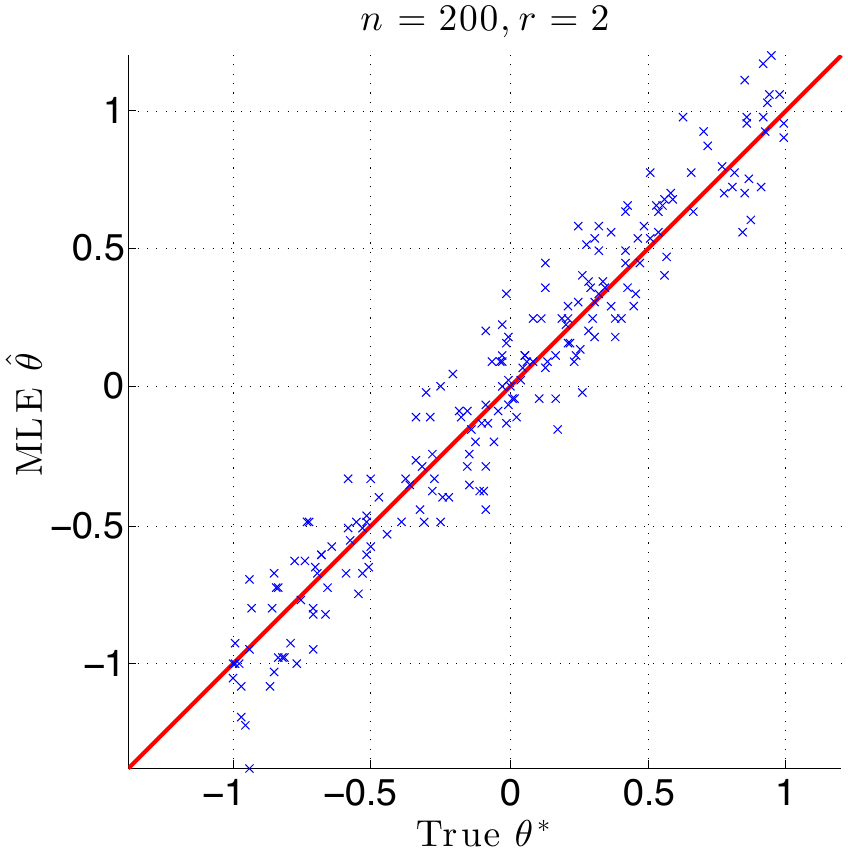} &
\widgraph{0.29\textwidth}{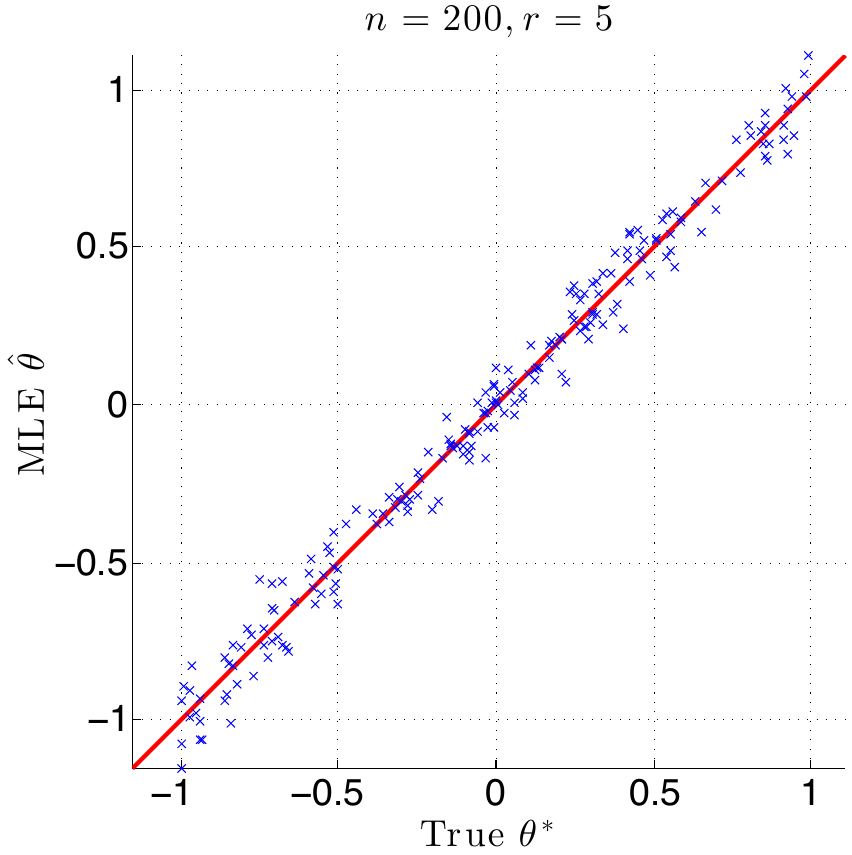} \\

(a) & (b) & (c)
\end{tabular}
\caption{(a) Plot of $\log \|\theta^{(t)} - \hat\theta\|_\infty$ for various values of $r$, where $\hat\theta$ is the final value of $\theta^{(t)}$ when the algorithm converges; (b) Scatter plot of the estimate $\hat \theta$ vs.\ the true parameter $\theta$ for $r = 2$; (c) Scatter plot for $r = 5$.}
\label{Fig:FiniteDisc}
\end{center}
\end{figure}

Figure~\ref{Fig:FiniteDisc} demonstrates the performance of the algorithm presented above. We set $n = 200$ and sample $\theta \in [-1,1]^n$ uniformly at random. Then for each $2 \leq r \leq 10$, we sample a graph from the distribution $\P^\ast_\theta$, compute the empirical degree sequence $\mathbf{d}$, and run the fixed-point algorithm starting with $\theta^{(0)} = \mathbf{0}$ until convergence. The left panel (Figure~\ref{Fig:FiniteDisc}(a)) shows the rate of convergence (on a logarithmic scale) of the algorithm for various values of $r$. We observe that the iterates $\{\theta^{(t)}\}$ indeed converge geometrically fast to the MLE $\hat \theta$, but the rate of convergence decreases as $r$ increases. By examining the proof of Theorem~\ref{Thm:MLEAlgFiniteDisc} in Section~\ref{Sec:ProofFiniteDisc}, we see that the term $\beta$ has the following expression:
\begin{equation*}
\beta^2 = 1-\frac{1}{(r-1)^2} \: \left(\min \left\{\frac{\exp(2K)-1}{\exp(2rK)-1}, \; -\frac{\mu'(2K)}{\mu(-2K)} \right\}\right)^2,
\end{equation*}
where $K = 2\|\hat\theta\|_\infty + \|\theta^{(0)}\|_\infty$. This shows that $\beta$ is an increasing function of $r$, explaining the empirical decrease in the rate of convergence as $r$ increases.

Figures~\ref{Fig:FiniteDisc}(b) and (c) show the plots of the estimate $\hat \theta$ versus the true $\theta$ (see also Figure~\ref{fig_imgs_app_1}). Notice that the points lie close to the diagonal line, suggesting that the MLE $\hat \theta$ is very close to the true parameter $\theta$. Indeed, the following result shows that $\hat \theta$ is a ``consistent estimator" of $\theta$ as $n \to \infty$ (see \cite{Chatterjee, rinaldo2013maximum} for proof details).

\begin{theorem}[Bounded discrete weights \cite{Chatterjee, rinaldo2013maximum}]\label{Thm:ConsistencyFiniteDisc}
Let $M > 0$ and $k > 1$ be fixed. Given $\theta \in \R^n$ with $\|\theta\|_\infty \leq M$, consider the problem of finding the MLE $\hat \theta$ of $\theta$ based on one graph sample $G \sim \P^\ast_\theta$. For sufficiently large $n$, with probability at least $1-2n^{-(k-1)}$, the MLE $\hat \theta$ exists and satisfies:
\begin{equation*}
\|\hat \theta - \theta\|_\infty \leq C \sqrt{\frac{k \log n}{n}},
\end{equation*}
where $C$ is a constant that only depends on $M$.
\end{theorem}

\subsection{Infinite discrete weighted graphs}
\label{Sec:InfiniteDisc}

We now turn our focus to weighted graphs with edge weights in $\N_0$. Full proofs of the results presented here can be found in Section~\ref{Sec:ProofInfiniteDisc}.

\subsubsection{Characterization of the distribution}

We take $\nu$ to be the counting measure on $\N_0$. In this case the marginal log-partition function is given by
\begin{equation*}
Z_1(t) = \log \sum_{a = 0}^\infty \exp(-at) = 
\begin{cases}
-\log\big( 1-\exp(-t) \big) & \text{ if } t > 0, \\
\infty \quad & \text{ if } t \leq 0.
\end{cases}
\end{equation*}
Thus, the domain of $Z_1$ is $\Dom(Z_1) = (0,\infty)$ with parameter space $\Theta = \{(\theta_1,\dots,\theta_n) \in \R^n \colon \theta_i+\theta_j > 0 \text{ for } i \neq j\}$,
the same as in the case of continuous weighted graphs. Given $\theta \in \Theta$, the edge weights $A_{ij}$ are independent geometric random variables with probability mass function:
\begin{equation*}
\P^\ast(A_{ij}=a) = \big(1-\exp(-\theta_i-\theta_j)\big) \: \exp\big( -(\theta_i+\theta_j) \: a \big), \quad a \in \N_0.
\end{equation*}
The mean parameters are:
\begin{equation*}
\E_{\P^\ast}[A_{ij}] = \frac{\exp(-\theta_i-\theta_j)}{1-\exp(-\theta_i-\theta_j)} = \frac{1}{\exp(\theta_i+\theta_j)-1},
\end{equation*}
induced by the mean function:
\begin{equation*}
\mu(t) = \frac{1}{\exp(t)-1}, \quad t > 0.
\end{equation*}

\subsubsection{Existence, uniqueness, and consistency of the MLE}

Consider the problem of finding the MLE of $\theta$ from one graph sample $G \sim \P^\ast_\theta$. Let $\mathbf{d}$ denote the degree sequence of $G$. Then the MLE $\hat \theta \in \Theta$, which satisfies the moment-matching equation $\E_{\hat \theta}[\deg(A)] = \mathbf{d}$, is a solution to the system of equations:
\begin{equation}\label{Eq:MLEEqInfiniteDisc}
d_i = \sum_{j \neq i} \frac{1}{\exp(\hat{\theta}_i + \hat{\theta}_j)-1}, \quad i = 1,\dots,n.
\end{equation}

As before, the MLE $\hat \theta$ exists and is unique if and only if $\mathbf{d} \in \M^\circ$, where $\M$ is the mean parameter space.
Since $\nu^{\binom{n}{2}}$ is the counting measure on $\N_0^{\binom{n}{2}}$, the set $\mathfrak{P}$ contains all the Dirac measures, so that $\M = \conv(\W)$ from Proposition~\ref{Prop:MConvW}. Here $\W$ is the set of all {\em (infinite discrete) graphical sequences}; namely, degree sequences of weighted graphs with edge weights in $\N_0$. The following result provides a precise criterion for such graphical sequences. Note that condition~\eqref{Eq:GraphicalInfiniteDisc} below is implied by the limit $r \to \infty$ in Theorem~\ref{Thm:GraphicalFiniteDisc}.

The criterion in Theorem~\ref{Thm:GraphicalInfiniteDisc} allows us to write an explicit form for the set of graphical sequences $\W$:
\begin{equation*}
\W = \Big\{(d_1, \dots, d_n) \in \N_0^n \colon \sum_{i=1}^n d_i \text{ is even and } \max_{1 \leq i \leq n} d_i \leq \frac{1}{2} \sum_{i=1}^n d_i \Big\}.
\end{equation*}
Now we need to characterize $\conv(\W)$. Let $\W_1$ denote the set of all continuous graphical sequences from Theorem~\ref{Thm:GraphicalCont}, when the edge weights are in $\R_0$:
\begin{equation*}
\W_1 = \Big\{(d_1, \dots, d_n) \in \R_0^n \colon \max_{1 \leq i \leq n} d_i \leq \frac{1}{2} \sum_{i=1}^n d_i \Big\}.
\end{equation*}
It turns out that when we take the convex hull of $\W$, we essentially recover $\W_1$.

\begin{lemma}\label{Lem:ConvWInfiniteDisc}
$\overline{\conv(\W)} = \W_1$.
\end{lemma}

Recalling that a convex set and its closure have the same interior points, the result above gives us:
\begin{equation*}
\M^\circ = \conv(\W)^\circ = \big(\overline{\conv(\W)}\big)^\circ = \W_1^\circ =
\Big\{(d_1, \dots, d_n) \in \R_+^n \colon \max_{1 \leq i \leq n} d_i < \frac{1}{2} \sum_{i=1}^n d_i \Big\}.
\end{equation*}

\begin{example}
Let $n = 3$ and $\mathbf{d} = (d_1,d_2,d_3) \in \R^n$ with $d_1 \geq d_2 \geq d_3$. It can be easily verified that the system of equations~\eqref{Eq:MLEEqInfiniteDisc} is:
\begin{equation*}
\begin{split}
\hat{\theta}_1+\hat{\theta}_2 &= \log\left( 1+\frac{2}{d_1+d_2-d_3} \right), \\
\hat{\theta}_1+\hat{\theta}_3 &= \log\left( 1+\frac{2}{d_1-d_2+d_3} \right), \\
\hat{\theta}_2+\hat{\theta}_3 &= \log\left( 1+\frac{2}{-d_1+d_2+d_3} \right),
\end{split}
\end{equation*}
from which we can obtain a unique solution $\hat{\theta} = (\hat{\theta}_1,\hat{\theta}_2,\hat{\theta}_3)$. Recall that $\hat{\theta} \in \Theta$ means $\hat{\theta}_1+\hat{\theta}_2>0$, $\hat{\theta}_1+\hat{\theta}_3>0$, and $\hat{\theta}_2+\hat{\theta}_3>0$, so the equations above tell us that $\hat\theta \in \Theta$ if and only if $2/(-d_1+d_2+d_3) > 0$, or equivalently, $d_1 < d_2+d_3$. This also implies $d_3 > d_1-d_2 \geq 0$, so $\mathbf{d} \in \R_+^3$. Thus, the system of equations~\eqref{Eq:MLEEqInfiniteDisc} has a unique solution $\hat{\theta} \in \Theta$ if and only if $\mathbf{d} \in \M^\circ$, as claimed above.
\end{example}

Finally, we prove that with high probability the MLE $\hat \theta$ exists and converges to $\theta$ as $n \to \infty$.

\begin{theorem}[Unbounded discrete weights]\label{Thm:ConsistencyInfiniteDisc}
Let $M \geq L > 0$ and $k \geq 1$ be fixed. Given $\theta \in \Theta$ with $L \leq \theta_i + \theta_j \leq M$, $i \neq j$, consider the problem of finding the MLE $\hat \theta \in \Theta$ of $\theta$ from one graph sample $G \sim \P^\ast_\theta$. Then for sufficiently large $n$, with probability at least $1-3n^{-(k-1)}$ the MLE $\hat \theta \in \Theta$ exists and satisfies:
\begin{equation*}
\|\hat \theta - \theta\|_\infty \leq \frac{8 \: \exp(5M)}{L} \: \sqrt{\frac{k \log n}{\gamma n}},
\end{equation*}
where $\gamma > 0$ is a universal constant.
\end{theorem}

\subsection{Continuous weighted graphs}
\label{Sec:Cont}

In this section we study weighted graphs with edge weights in $\R_0$. The proofs of the results presented here are provided in Section~\ref{Sec:ProofCont}.

\subsubsection{Characterization of the distribution}

We take $\nu$ to be the Lebesgue measure on $\R_0$. The marginal log-partition function is:
\begin{equation*}
Z_1(t) = \log \int_{\R_0} \exp(-ta) \: da =
\begin{cases}
\log(1/t) & \text{ if } t > 0 \\
\infty \quad & \text{ if } t \leq 0.
\end{cases}
\end{equation*}
Thus, $\Dom(Z_1) = \R_+$, and the natural parameter space is $\Theta = \{(\theta_1, \dots, \theta_n) \in \R^n \colon \theta_i+\theta_j > 0 \text{ for } i \neq j\}$.
For $\theta \in \Theta$, the edge weights $A_{ij}$ are independent exponential random variables with density:
\begin{equation*}
p_{ij}^\ast(a) = (\theta_i+\theta_j) \: \exp\big(-(\theta_i+\theta_j) \: a\big), \quad \text{ for } a \in \R_0,
\end{equation*}
and mean parameter $\E_\theta[A_{ij}] = 1/(\theta_i+\theta_j)$. The corresponding mean function is given by $\mu(t) = \frac{1}{t}, \, t > 0$.

\subsubsection{Existence, uniqueness, and consistency of the MLE}

We now consider the problem of finding the MLE of $\theta$ from one graph sample $G \sim \P^\ast_\theta$. As we saw previously, the MLE $\hat \theta \in \Theta$ satisfies the moment-matching equation $\E_{\hat \theta}[\deg(A)] = \mathbf{d}$, where $\mathbf{d}$ is the degree sequence of the sample graph $G$. Equivalently, $\hat \theta \in \Theta$ is a solution to the following system of equations:
\begin{equation}\label{Eq:MLEEqCont}
d_i = \sum_{j \neq i} \frac{1}{\hat{\theta}_i + \hat{\theta}_j}, \quad i = 1,\dots,n.
\end{equation}

\begin{remark}\label{SturmfelsRemark}
The system~\eqref{Eq:MLEEqCont} is a special case of a general class that Sanyal, Sturmfels, and Vinzant~\cite{SturmEntDisc} study using algebraic geometry and matroid theory.  Consider the following polynomial in $t$:
\begin{equation*}
\chi(t) = \sum_{k=0}^n \left(\stirlingtwo{n}{k} +n \stirlingtwo{n-1}{k} \right)(t-1)_k^{(2)},
\end{equation*}
in which $\stirlingtwo{n}{k}$ is the Stirling number of the second kind and $(x)_{k+1}^{(2)} = x(x-2)\cdots (x-2k)$ is a generalized falling factorial.  Then, there is a polynomial $H(\mathbf{d})$ in the $d_i$ such that for $\mathbf{d} \in \mathbb R^n$ with $H(\mathbf{d}) \neq 0$, the number of solutions $\theta \in \mathbb R^n$ to~\eqref{Eq:MLEEqCont} is $(-1)^n \chi(0)$.  Moreover, the polynomial $H(\mathbf{d})$ has degree $2(-1)^n(n \chi(0) + \chi'(0))$ and characterizes those $\mathbf{d}$ for which the equations above have multiple roots.  We refer to \cite{SturmEntDisc} for more details.
\end{remark}

The MLE $\hat \theta$ exists and is unique if and only if the empirical degree sequence $\mathbf{d}$ belongs to the interior $\M^\circ$.  We characterize the set of graphical sequences $\W$ and determine its relation to $\M$.  The finite discrete graphical sequences from Section~\ref{Sec:FiniteDiscrete} have combinatorial constraints because there are only finitely many possible edge weights between any pair of vertices, and these constraints translate into a set of inequalities in the generalized Erd\H{o}s-Gallai criterion in Theorem~\ref{Thm:GraphicalFiniteDisc}. In the case of continuous weighted graphs, however, we do not have these constraints because every edge can have as much weight as possible. Therefore, the criterion for a continuous graphical sequence should be simpler than in Theorem~\ref{Thm:GraphicalFiniteDisc}, as the following  shows.

\begin{theorem}\label{Thm:GraphicalCont}
A sequence $(d_1, \dots, d_n) \in \R_0^n$ is graphic if and only if:
\begin{equation}\label{Eq:GraphicalCont}
\max_{1 \leq i \leq n} d_i \leq \frac{1}{2} \sum_{i=1}^n d_i.
\end{equation}
\end{theorem}

We note that condition~\eqref{Eq:GraphicalCont} is implied by the case $k = 1$ in the conditions~\eqref{Eq:GraphicalFiniteDisc}. This is to be expected, since any finite discrete weighted graph is also a continuous weighted graph, so finite discrete graphical sequences are also continuous graphical sequences.

Given the criterion in Theorem~\ref{Thm:GraphicalCont}, we can write the set $\W$ of graphical sequences as follows:
\begin{equation*}
\W = \Big\{(d_1, \dots, d_n) \in \R_0^n \colon \max_{1 \leq i \leq n} d_i \leq \frac{1}{2} \sum_{i=1}^n d_i \Big\}.
\end{equation*}
Moreover, we can also show that the set of graphical sequences coincide with the mean parameter space.

\begin{lemma}\label{Lem:W-Convex}
The set $\W$ is convex, and $\M = \W$.
\end{lemma}

The result above, together with the result of Proposition~\ref{Prop:RegularMinimal}, implies that the MLE $\hat \theta$ exists and is unique if and only if the empirical degree sequence $\mathbf{d}$ belongs to the interior of the mean parameter space, which can be written explicitly:
\begin{equation*}
\M^\circ = \Big\{(d_1', \dots, d_n') \in \R_+^n \colon \max_{1 \leq i \leq n} d_i' < \frac{1}{2} \sum_{i=1}^n d_i' \Big\}.
\end{equation*}

\begin{example}
Let $n = 3$ and $\mathbf{d} = (d_1,d_2,d_3) \in \R^n$ with $d_1 \geq d_2 \geq d_3$. It is easy to see that the system of equations~\eqref{Eq:MLEEqCont} reduces to:
\begin{equation*}
\begin{split}
\frac{1}{\hat{\theta}_1+\hat{\theta}_2} &= \frac{1}{2}(d_1+d_2-d_3), \\
\frac{1}{\hat{\theta}_1+\hat{\theta}_3} &= \frac{1}{2}(d_1-d_2+d_3), \\
\frac{1}{\hat{\theta}_2+\hat{\theta}_3} &= \frac{1}{2}(-d_1+d_2+d_3),
\end{split}
\end{equation*}
from which we obtain a unique solution $\hat{\theta} = (\hat{\theta}_1,\hat{\theta}_2,\hat{\theta}_3)$. Recall that $\hat{\theta} \in \Theta$ means $\hat{\theta}_1+\hat{\theta}_2 > 0$, $\hat{\theta}_1+\hat{\theta}_3 > 0$, and $\hat{\theta}_2+\hat{\theta}_3 > 0$, so the equations above tell us that $\hat{\theta} \in \Theta$ if and only if $d_1 < d_2+d_3$. In particular, this also implies $d_3 > d_1-d_2 \geq 0$, so $\mathbf{d} \in \R_+^3$. Hence, there is a unique solution $\hat{\theta} \in \Theta$ to the system of equations~\eqref{Eq:MLEEqCont} if and only if $\mathbf{d} \in \M^\circ$, as claimed above.
\end{example}

Finally, we prove that the MLE $\hat \theta$ is a consistent estimator of $\theta$ as $n \to \infty$.

\begin{theorem}[Unbounded continuous weights]\label{Thm:ConsistencyCont}
Let $M \geq L > 0$ and $k \geq 1$ be fixed. Given $\theta \in \Theta$ with $L \leq \theta_i + \theta_j \leq M$, $i \neq j$, consider the problem of finding the MLE $\hat \theta \in \Theta$ of $\theta$ from one graph sample $G \sim \P^\ast_\theta$. Then for sufficiently large $n$, with probability at least $1-2n^{-(k-1)}$ the MLE $\hat \theta \in \Theta$ exists and satisfies:
\begin{equation*}
\|\hat \theta - \theta\|_\infty \leq \frac{100M^2}{L} \sqrt{\frac{k \log n}{\gamma n}},
\end{equation*}
where $\gamma > 0$ is a universal constant.
\end{theorem}

\section{Proofs of main results}
\label{Sec:Proofs}

In this section we provide proofs for the technical results presented in Section~\ref{Sec:Specific}. The proofs of the characterization of weighted graphical sequences (Theorems~\ref{Thm:GraphicalFiniteDisc},~\ref{Thm:GraphicalCont}, and~\ref{Thm:GraphicalInfiniteDisc}) are inspired by the constructive proof of the classical Erd\H{o}s-Gallai criterion by Choudum~\cite{Choudum}.

\subsection{Preliminaries}

We begin by presenting several results that we will use in this section. We use the definition of sub-exponential random variables and the concentration inequality presented in~\cite{vershynin}.

\subsubsection{Concentration inequality for sub-exponential random variables}

We say that a real-valued random variable $X$ is {\em sub-exponential} with parameter $\kappa > 0$ if $\E[|X|^p]^{1/p} \leq \kappa p, \text{ for all } p \geq 1$.
Note that if $X$ is a $\kappa $-sub-exponential random variable with finite first moment, then the centered random variable $X-\E[X]$ is also sub-exponential with parameter $2 \kappa $. This follows from the triangle inequality applied to the $p$-norm, followed by Jensen's inequality for $p \geq 1$:
\begin{equation*}
\begin{split}
\E\big[\big|X-\E[X]\big|^p\big]^{1/p}
&\leq \E[|X|^p]^{1/p} + \big|\E[X]\big|
\leq 2\E[|X|^p]^{1/p}.
\end{split}
\end{equation*}
Sub-exponential random variables satisfy the following concentration inequality.

\begin{theorem}[{\cite[Corollary~5.17]{vershynin}}]\label{Thm:ConcIneqSubExp}
Let $X_1, \dots, X_n$ be independent centered random variables, and suppose each $X_i$ is sub-exponential with parameter $\kappa_i$. Let $\kappa = \max_{1 \leq i \leq n} \kappa_i$. For every $\epsilon \geq 0$, we have the bound:
\begin{equation*}
\P\left( \left| \frac{1}{n} \sum_{i=1}^n X_i \right| \geq \epsilon \right) \leq 2\exp\left[-\gamma \, n \cdot \min\Big(\frac{\epsilon^2}{\kappa^2}, \: \frac{\epsilon}{\kappa} \Big) \right],
\end{equation*}
with $\gamma > 0$ is an absolute constant.
\end{theorem}

We will apply the concentration inequality above to exponential and geometric random variables, which are the distributions of the edge weights of continuous weighted graphs (from Section~\ref{Sec:Cont}) and infinite discrete weighted graphs (from Section~\ref{Sec:InfiniteDisc}).

\begin{lemma}\label{Lem:SubExp-Exp}
Let $X$ be an exponential random variable with $\E[X] = 1/\lambda$. Then $X$ is sub-exponential with parameter $1/\lambda$, and the centered random variable $X-1/\lambda$ is sub-exponential with parameter $2/\lambda$.
\end{lemma}
\begin{proof}
For any $p \geq 1$, we can evaluate the moment of $X$ directly:
\begin{equation*}
\E[|X|^p] = \int_0^\infty x^p \cdot \lambda \, \exp(-\lambda x) \: dx
= \frac{1}{\lambda^p} \: \int_0^\infty y^p \: \exp(-y) \: dy
= \frac{\Gamma(p+1)}{\lambda^p},
\end{equation*}
where $\Gamma$ is the gamma function, and in the computation above we have used the substitution $y = \lambda x$. It can be easily verified that $\Gamma(p+1) \leq p^p$ for $p \geq 1$, so that:
\begin{equation*}
\E[|X|^p]^{1/p} = \frac{\big(\Gamma(p+1)\big)^{1/p}}{\lambda} \leq \frac{p}{\lambda}.
\end{equation*}
This shows that $X$ is sub-exponential with parameter $1/\lambda$.
\end{proof}

\begin{lemma}\label{Lem:SubExp-Geo}
Let $X$ be a geometric random variable with parameter $q \in (0,1)$, so that:
\begin{equation*}
\P(X = a) = (1-q)^a \, q, \quad a \in \N_0.
\end{equation*}
Then $X$ is sub-exponential with parameter $-2/\log(1-q)$, and the centered random variable $X - (1-q)/q$ is sub-exponential with parameter $-4/\log(1-q)$.
\end{lemma}
\begin{proof}
Fix $p \geq 1$, and consider the function $f \colon \R_0 \to \R_0$, $f(x) = x^p (1-q)^x$. One can easily verify that $f$ is increasing for $0 \leq x \leq \lambda$ and decreasing on $x \geq \lambda$, where $\lambda = -p/\log(1-q)$. In particular, for all $x \in \R_0$ we have $f(x) \leq f(\lambda)$, and:
\begin{equation*}
\begin{split}
f(\lambda) &= \lambda^p (1-q)^\lambda
= \left( \frac{p}{-\log (1-q)} \cdot (1-q)^{-1/\log(1-q)} \right)^p
= \left( \frac{p}{-e \cdot \log (1-q)} \right)^p.
\end{split}
\end{equation*}

Now note that for $0 \leq a \leq \lfloor \lambda \rfloor - 1$ we have $f(a) \leq \int_a^{a+1} f(x) \: dx$, and for $a \geq \lceil \lambda \rceil + 1$ we have $f(a) \leq \int_{a-1}^a f(x) \: dx$. Thus, we can make 
the estimate:
\begin{equation*}
\begin{split}
\sum_{a=0}^\infty f(a)
&= \sum_{a=0}^{\lfloor \lambda \rfloor - 1} f(a) + \sum_{a=\lfloor \lambda \rfloor}^{\lceil \lambda \rceil} f(a) + \sum_{a = \lceil \lambda \rceil + 1}^\infty f(a) \\
&\leq \int_0^{\lfloor \lambda \rfloor} f(x) \: dx + 2f(\lambda) + \int_{\lceil \lambda \rceil}^\infty f(x) \: dx \\
&\leq \int_0^\infty f(x) \: dx + 2f(\lambda).
\end{split}
\end{equation*}
Using the substitution $y = -x \log(1-q)$, the integral can be calculated as:
\begin{equation*}
\begin{split}
\int_0^\infty f(x) \: dx
&= \int_0^\infty x^p \exp\left(x \cdot \log (1-q) \right) \: dx
= \frac{1}{(-\log(1-q))^{p+1}} \: \int_0^\infty y^p \: \exp(-y) \: dy \\
&= \frac{\Gamma(p+1)}{(-\log(1-q))^{p+1}}
\leq \frac{p^p}{(-\log(1-q))^{p+1}},
\end{split}
\end{equation*}
where in the last step we have again used the relation $\Gamma(p+1) \leq p^p$.
We use the result above, along with the expression of $f(\lambda)$, to bound the moment of $X$:
\begin{equation*}
\begin{split}
\E[|X|^p] &= \sum_{a=0}^\infty a^p \cdot (1-q)^a \: q
= q \: \sum_{a=0}^\infty f(a) \\
&\leq q \: \int_0^\infty f(x) \: dx + 2q \: f(\lambda) \\
&\leq \left(\frac{q^{1/p} \: p}{(-\log(1-q))^{1+1/p}} \right)^p + \left( \frac{(2q)^{1/p} \: p}{-e \cdot \log (1-q)} \right)^p \\
&\leq \left( \frac{q^{1/p} \: p}{(-\log(1-q))^{1+1/p}} + \frac{(2q)^{1/p} \: p}{-e \cdot \log (1-q)} \right)^p,
\end{split}
\end{equation*}
where in the last step we have used the fact that $x^p + y^p \leq (x+y)^p$ for $x,y \geq 0$ and $p \geq 1$. This gives us
\begin{equation*}
\begin{split}
\frac{1}{p} \: \E[|X|^p]^{1/p}
&\leq \frac{q^{1/p}}{(-\log(1-q))^{1+1/p}} + \frac{2^{1/p} \: q^{1/p}}{-e \cdot \log (1-q)} \\
&= \frac{1}{-\log(1-q)} \: \left(\left(\frac{q}{-\log(1-q)}\right)^{1/p} + \frac{2^{1/p} \: q^{1/p}}{e}\right).
\end{split}
\end{equation*}
Next, note that $q \leq -\log(1-q)$ for $0 < q < 1$, so $(-q/\log(1-q))^{1/p} \leq 1$. Moreover, $(2q)^{1/p} \leq 2^{1/p} \leq 2$. Therefore, for any $p \geq 1$, we have:
\begin{equation*}
\begin{split}
\frac{1}{p} \: \E[|X|^p]^{1/p}
&\leq \frac{1}{-\log(1-q)} \: \left(1 + \frac{2}{e}\right)
< \frac{2}{-\log(1-q)}.
\end{split}
\end{equation*}
We conclude that $X$ is sub-exponential with parameter $-2/\log(1-q)$.
\end{proof}

\subsubsection{Bound on the inverses of diagonally-dominant matrices}

An $n\times n$ real matrix $J$ is {\em diagonally dominant} if:
\begin{equation*}
\Delta_i(J) := |J_{ii}| - \sum_{j \neq i} |J_{ij}| \geq 0, \quad \text{for } i = 1,\dots,n.
\end{equation*}
We say that $J$ is {\em diagonally balanced} if $\Delta_i(J) = 0$ for $i = 1,\dots,n$. We have the following bound from~\cite{HLW} on the inverses of diagonally dominant matrices. This bound is independent of $\Delta_i$, so it is also applicable to diagonally balanced matrices. We will use this result in the proofs of Theorems~\ref{Thm:ConsistencyCont} and~\ref{Thm:ConsistencyInfiniteDisc}.

\begin{theorem}[{\cite[Theorem~1.1]{HLW}}]\label{Thm:Main}
Let $n \geq 3$. For any symmetric diagonally dominant matrix $J$ with $J_{ij} \geq \ell > 0$, we have the inequality:
\begin{equation*}
\|J^{-1}\|_\infty \leq \frac{3n-4}{2\ell(n-2)(n-1)}.
\end{equation*}
\end{theorem}

\subsection{Proofs for the finite discrete weighted graphs}
\label{Sec:ProofFiniteDisc}

In this section we present the proofs of the results presented in Section~\ref{Sec:FiniteDiscrete}.

\subsubsection{Proof of Theorem~\ref{Thm:GraphicalFiniteDisc}}

We first prove the necessity of~\eqref{Eq:GraphicalFiniteDisc}. Suppose $\mathbf{d} = (d_1,\dots,d_n)$ is the degree sequence of a graph $G$ with edge weights $a_{ij} \in S$. Then $\sum_{i=1}^n d_i = 2\sum_{(i,j)} a_{ij}$ is even. Moreover, for each $1 \leq k \leq n$, $\sum_{i=1}^k d_i$ counts the total edge weights coming out from the vertices $1,\dots,k$. The total edge weights from these $k$ vertices to themselves is at most $(r-1)k(k-1)$, and for each vertex $j \notin \{1,\dots,k\}$, the total edge weights from these $k$ vertices to vertex $j$ is at most $\min\{d_j,(r-1)k\}$, so by summing over $j \notin \{1,\dots,k\}$ we get~\eqref{Eq:GraphicalFiniteDisc}.

To prove the sufficiency of~\eqref{Eq:GraphicalFiniteDisc} we use induction on $s := \sum_{i=1}^n d_i$. The base case $s = 0$ is trivial. Assume the statement holds for $s-2$, and suppose we have a sequence $\mathbf{d}$ with $d_1 \geq d_2 \geq \dots \geq d_n$ satisfying~\eqref{Eq:GraphicalFiniteDisc} with $\sum_{i=1}^n d_i = s$. Without loss of generality we may assume $d_n \geq 1$, for otherwise we can proceed with only the nonzero elements of $\mathbf{d}$. Let $1 \leq t \leq n-1$ be the smallest index such that $d_t > d_{t+1}$, with $t = n-1$ if $d_1 = \cdots = d_n$. Define $\mathbf{d}' = (d_1,\dots,d_{t-1},d_t-1,d_{t+1},\dots,d_{n-1},d_n-1)$, so we have $d_1' = \cdots = d_{t-1}' > d_t' \geq d_{t+1}' \geq \cdots \geq d_{n-1}' > d_n'$ and $\sum_{i=1}^n d_i' = s-2$.

We will show that $\mathbf{d}'$ satisfies~\eqref{Eq:GraphicalFiniteDisc}. By the inductive hypothesis, this means $\mathbf {d}'$ is the degree sequence of a graph $G'$ with edge weights $a_{ij}' \in \{0,1,\dots,r-1\}$. We now attempt to modify $G'$ to obtain a graph $G$ whose degree sequence is equal to $\mathbf{d}$. If the weight $a_{tn}'$ of the edge $(t,n)$ is less than $r-1$, then we can obtain $G$ by increasing $a_{tn}'$ by $1$, since the degree of vertex $t$ is now $d_t'+1 = d_t$, and the degree of vertex $n$ is now $d_n'+1 = d_n$. Otherwise, suppose $a_{tn}' = r-1$. Since $d_t' = d_1'-1$, there exists a vertex $u \neq n$ such that $a_{tu}' < r-1$. Since $d_u' > d_n'$, there exists another vertex $v$ such that $a_{uv}' > a_{vn}'$. Then we can obtain the graph $G$ by increasing $a_{tu}'$ and $a_{vn}'$ by $1$ and reducing $a_{uv}'$ by $1$, so that now the degrees of vertices $t$ and $n$ are each increased by $1$, and the degrees of vertices $u$ and $v$ are preserved.

It now remains to show that $\mathbf{d}'$ satisfies~\eqref{Eq:GraphicalFiniteDisc}. We divide the proof into several cases for different values of $k$. We will repeatedly use the fact that $\mathbf{d}$ satisfies~\eqref{Eq:GraphicalFiniteDisc}, as well as the inequality $\min\{a,b\}-1 \leq \min\{a-1,b\}$.
\begin{enumerate}
  \item For $k = n$:
\begin{equation*}
\sum_{i=1}^n d_i' = \sum_{i=1}^n d_i-2 \leq (r-1)n(n-1)-2 < (r-1)n(n-1).
\end{equation*}
  \item For $t \leq k \leq n-1$:
\begin{equation*}
\begin{split}
\sum_{i=1}^k d_i' &= \sum_{i=1}^k d_i-1 \leq (r-1)k(k-1) + \sum_{j=k+1}^n \min\{d_j,(r-1)k\} - 1 \\
&\leq (r-1)k(k-1) + \sum_{j=k+1}^{n-1} \min\{d_j,(r-1)k\} + \min\{d_n-1,(r-1)k\} \\
&= (r-1)k(k-1) + \sum_{j=k+1}^n \min\{d_j',(r-1)k\}.
\end{split}
\end{equation*}
  \item For $1 \leq k \leq t-1$: first suppose $d_n \geq 1+(r-1)k$. Then for all $j$ we have:
\begin{equation*}
\min\{d_j',(r-1)k\} = \min\{d_j,(r-1)k\} = (r-1)k,
\end{equation*}  
so that:
\begin{equation*}
\begin{split}
\sum_{i=1}^k d_i' &= \sum_{i=1}^k d_i \leq (r-1)k(k-1) + \sum_{j=k+1}^n \min\{d_j,(r-1)k\} \\
&= (r-1)k(k-1) + \sum_{j=k+1}^n \min\{d_j',(r-1)k\}.
\end{split}
\end{equation*}
  \item For $1 \leq k \leq t-1$: suppose $d_1 \geq 1 + (r-1)k$, and $d_n \leq (r-1)k$. We claim that $\mathbf{d}$ satisfies~\eqref{Eq:GraphicalFiniteDisc} at $k$ with a strict inequality. If this claim is true, then, since $d_t = d_1$ and $\min\{d_t',(r-1)k\} = \min\{d_t,(r-1)k\} = (r-1)k$, it follows that:
\begin{equation*}
\begin{split}
\sum_{i=1}^k d_i' &= \sum_{i=1}^k d_i
\leq (r-1)k(k-1) + \sum_{j=k+1}^n \min\{d_j,(r-1)k\}-1 \\
&= (r-1)k(k-1) + \sum_{j=k+1}^{n-1} \min\{d_j',(r-1)k\} + \min\{d_n,(r-1)k\}-1 \\
&\leq (r-1)k(k-1) + \sum_{j=k+1}^{n-1} \min\{d_j',(r-1)k\} + \min\{d_n-1,(r-1)k\} \\
&= (r-1)k(k-1) + \sum_{j=k+1}^n \min\{d_j',(r-1)k\}.
\end{split}
\end{equation*}
Now to prove the claim, suppose the contrary that $\mathbf{d}$ satisfies~\eqref{Eq:GraphicalFiniteDisc} at $k$ with equality. Let $t+1 \leq u \leq n$ be the smallest integer such that $d_u \leq (r-1)k$. Our assumption implies:
\begin{equation*}
\begin{split}
k d_k &= \sum_{i=1}^k d_i = (r-1)k(k-1) + \sum_{j=k+1}^n \min\{d_j,(r-1)k\} \\
&\geq (r-1)k(k-1) + (u-k-1)(r-1)k + \sum_{j=u}^n d_j \\
&= (r-1)k(u-2) + \sum_{j=u}^n d_j.
\end{split}
\end{equation*}
Therefore, since $d_{k+1} = d_k = d_1$, we have:
\begin{equation*}
\begin{split}
\sum_{i=1}^{k+1} d_i &= (k+1)d_k
\geq (r-1)(k+1)(u-2) + \frac{k+1}{k} \sum_{j=u}^n d_j \\
&> (r-1)(k+1)k + (r-1)(k+1)(u-k-2) + \sum_{j=u}^n d_j \\
&\geq (r-1)(k+1)k + \sum_{j=k+2}^n \min\{d_j, (r-1)(k+1)\},
\end{split}
\end{equation*}
which contradicts the fact that $\mathbf{d}$ satisfies~\eqref{Eq:GraphicalFiniteDisc} at $k+1$. Thus, we have proved that $\mathbf{d}$ satisfies~\eqref{Eq:GraphicalFiniteDisc} at $k$ with a strict inequality.

  \item For $1 \leq k \leq t-1$: suppose $d_1 \leq (r-1)k$. In particular, we have $\min\{d_j,(r-1)k\} = d_j$ and $\min\{d_j',(r-1)k\} = d_j'$ for all $j$. First, if we have:
\begin{equation}\label{Eq:ErdosGallaiProofCond}
d_{k+2} + \cdots + d_n \geq 2,
\end{equation}
then we are done, since:
\begin{equation*}
\begin{split}
\sum_{i=1}^k d_i' &= \sum_{i=1}^k d_i
= (k-1)d_1 + d_{k+1} \\
&\leq (r-1)k(k-1) + d_{k+1} + d_{k+2} + \cdots + d_n - 2 \\
&= (r-1)k(k-1) + \sum_{j=k+1}^n d_j' \\
&= (r-1)k(k-1) + \sum_{j=k+1}^n \min\{d_j',(r-1)k\}.
\end{split}
\end{equation*}
Condition~\eqref{Eq:ErdosGallaiProofCond} is obvious if $d_n \geq 2$ or $k+2 \leq n-1$ (since there are $n-k-1$ terms in the summation and each term is at least $1$). Otherwise, assume $k+2 \geq n$ and $d_n = 1$, so in particular, we have $k = n-2$ (since $k \leq t-1 \leq n-2$), $t = n-1$, and $d_1 \leq (r-1)(n-2)$. Note that we cannot have $d_1 = (r-1)(n-2)$, for then $\sum_{i=1}^n d_i = (n-1)d_1 + d_n = (r-1)(n-1)(n-2) + 1$ would be odd, so we must have $d_1 < (r-1)(n-2)$. Similarly, $n$ must be even, for otherwise $\sum_{i=1}^n d_i = (n-1)d_1 + 1$ would be odd. Thus, since $1 \leq d_1 < (r-1)(n-2)$ we must have $n \geq 4$. Therefore, it follows that:
\begin{equation*}
\begin{split}
\sum_{i=1}^k d_i' &= (n-2)d_1 = (n-3)d_1 + d_{n-1} \\
&\leq (r-1)(n-2)(n-3) - (n-3) + d_{n-1} \\
&\leq (r-1)(n-2)(n-3) + (d_{n-1}-1) + (d_n-1) \\
&= (r-1)(n-2)(n-3) + \sum_{j=k+1}^n \min\{d_j',(r-1)k\}.
\end{split}
\end{equation*}
\end{enumerate}
This shows that $\mathbf{d}'$ satisfies~\eqref{Eq:GraphicalFiniteDisc} and finishes the proof of Theorem~\ref{Thm:GraphicalFiniteDisc}.

\subsubsection{Proof of Theorem~\ref{Thm:MLEAlgFiniteDisc}}

We follow the outline of the proof of~\cite[Theorem~1.5]{Chatterjee}, omitting straightforward verifications. We first present the following properties of the mean function $\mu(t)$ and the Jacobian matrix of the function $\varphi$~\eqref{Eq:VarphiFiniteDisc}. We then combine these results at the end of this section into a proof of Theorem~\ref{Thm:MLEAlgFiniteDisc}.

\begin{lemma}\label{Lem:PropertyMeanFunction}
The mean function $\mu(t)$ is positive and strictly decreasing, with $\mu(-t) + \mu(t) = r-1$ for all $t \in \R$, and $\mu(t) \to 0$ as $t \to \infty$. Its derivative $\mu'(t)$ is increasing for $t \geq 0$, with the properties that $\mu'(t) < 0$, $\mu'(t) = \mu'(-t)$ for all $t \in \R$, and $\mu'(0) = -(r^2-1)/12$.
\end{lemma}
\begin{proof}
It is clear from~\eqref{Eq:MeanFuncFiniteDiscrete} that $\mu(t)$ is positive. From the alternative representation~\eqref{Eq:MeanFuncFiniteDiscreteAlt} it is easy to see that $\mu(-t) + \mu(t) = r-1$, and $\mu(t) \to 0$ as $t \to \infty$. Differentiating expression~\eqref{Eq:MeanFuncFiniteDiscrete} yields the following formula:
\begin{equation*}
\mu'(t)
= \frac{-(\sum_{a=0}^{r-1} a^2 \, \exp(-at))(\sum_{a=0}^{r-1} \exp(-at)) + (\sum_{a=0}^{r-1} a \: \exp(-at))^2}{(\sum_{a=0}^{r-1} \exp(-at))^2},
\end{equation*}
and substituting $t = 0$ gives us $\mu'(0) = -(r^2-1)/12$. The Cauchy-Schwarz inequality applied to the expression above tells us that $\mu'(t) < 0$, where the inequality is strict because the vectors $(a^2 \exp(-at))_{a=0}^{r-1}$ and $(\exp(-at))_{a=0}^{r-1}$ are not linearly dependent. Thus, $\mu(t)$ is strictly decreasing for all $t \in \R$.

The relation $\mu(-t) + \mu(t) = r-1$ gives us $\mu'(-t) = \mu'(t)$. Furthermore, by differentiating the expression~\eqref{Eq:MeanFuncFiniteDiscreteAlt} twice, one can verify that $\mu''(t) \geq 0$ for $t \geq 0$, which means $\mu'(t)$ is increasing for $t \geq 0$. See also Figure~\ref{Fig:FiniteDiscMean} for the behavior of $\mu(t)$ and $\mu'(t)$ for different values of $r$.
\end{proof}

\begin{lemma}\label{Lem:LowerBoundRatio}
For all $t \in \R$, we have $\frac{\mu'(t)}{\mu(t)} \geq -r+1 + \frac{1}{\sum_{a=0}^{r-1} \exp(at)} > -r+1$.
\end{lemma}

We recall the following definition and result from~\cite{Chatterjee}. Given $\delta > 0$, let $\mathcal{L}_n(\delta)$ denote the set of $n \times n$ matrices $A = (a_{ij})$ with $\|A\|_\infty \leq 1$, $a_{ii} \geq \delta$, and $a_{ij} \leq -\delta/(n-1)$, for each $1 \leq i \neq j \leq n$.

\begin{lemma}[{\cite[Lemma~2.1]{Chatterjee}}]\label{Lem:LnDelta}
If $A,B \in \mathcal{L}_n(\delta)$, then:
\begin{equation*}
\|AB\|_\infty \leq 1-\frac{2(n-2)\delta^2}{(n-1)}.
\end{equation*}
In particular, for $n \geq 3$,
\begin{equation*}
\|AB\|_\infty \leq 1-\delta^2.
\end{equation*}
\end{lemma}

Given $\theta,\theta' \in \R^n$, let $J(\theta,\theta')$ denote the $n \times n$ matrix whose $(i,j)$-entry is:
\begin{equation}\label{Eq:JacobianMLEDef}
J_{ij}(\theta,\theta') = \int_0^1 \frac{\partial \varphi_i}{\partial \theta_j} (t\theta + (1-t)\theta') \: dt.
\end{equation}

\begin{lemma}\label{Lem:NormJInfinity}
For all $\theta,\theta' \in \R^n$, we have $\|J(\theta,\theta')\|_\infty = 1$.
\end{lemma}
\begin{proof}
The partial derivatives of $\varphi$~\eqref{Eq:VarphiFiniteDisc} are:
\begin{equation}\label{Eq:PartialVarphi}
\frac{\partial \varphi_i(\mathbf{x})}{\partial x_i} = 1 + \frac{1}{(r-1)}\frac{\sum_{j \neq i} \mu'(x_i+x_j)}{\sum_{j \neq i} \mu(x_i+x_j)},
\end{equation}
and for $i \neq j$:
\begin{equation}\label{Eq:PartialVarphiMixed}
\frac{\partial \varphi_i(\mathbf{x})}{\partial x_j} = \frac{1}{(r-1)}\frac{\mu'(x_i+x_j)}{\sum_{k \neq i} \mu(x_i+x_k)} < 0,
\end{equation}
where the last inequality follows from $\mu'(x_i + x_j) < 0$. Using the result of Lemma~\ref{Lem:LowerBoundRatio} and the fact that $\mu$ is positive, we also see that:
\begin{equation*}
\begin{split}
\frac{\partial \varphi_i(\mathbf{x})}{\partial x_i}
&= 1 + \frac{1}{(r-1)} \frac{\sum_{j \neq i} \mu'(x_i+x_j)}{\sum_{j \neq i} \mu(x_i+x_j)}
> 1 + \frac{1}{(r-1)} \frac{\sum_{j \neq i} (-r+1) \mu(x_i+x_j)}{\sum_{j \neq i} \mu(x_i+x_j)}
= 0.
\end{split}
\end{equation*}
Setting $\mathbf{x} = t\theta + (1-t)\theta'$ and integrating over $0 \leq t \leq 1$, we also get that $J_{ij}(\theta,\theta') < 0$ for $i \neq j$, and $J_{ii}(\theta,\theta') = 1 + \sum_{j \neq i} J_{ij}(\theta,\theta') > 0$. This implies $\|J(\theta,\theta')\|_\infty = 1$, as desired.
\end{proof}

\begin{lemma}\label{Lem:JInLnDelta}
Let $\theta,\theta' \in \R^n$ with $\|\theta\|_\infty \leq K$ and $\|\theta'\|_\infty \leq K$ for some $K > 0$. Then $J(\theta,\theta') \in \mathcal{L}_n(\delta)$, where:
\begin{equation}\label{Eq:DeltaBoundJ}
\delta = \frac{1}{(r-1)} \: \min \left\{\frac{\exp(2K)-1}{\exp(2rK)-1}, \; -\frac{\mu'(2K)}{\mu(-2K)} \right\}.
\end{equation}
\end{lemma}
\begin{proof}
From Lemma~\ref{Lem:NormJInfinity} we already know that $J \equiv J(\theta,\theta')$ satisfies $\|J\|_\infty = 1$, so to show that $J\in \mathcal{L}_n(\delta)$ it remains to show that $J_{ii} \geq \delta$ and $J_{ij} \leq -\delta/(n-1)$ for $i \neq j$. In particular, it suffices to show that for each $0 \leq t \leq 1$ we have $\partial \varphi_i(\mathbf{x})/\partial x_i \geq \delta$ and $\partial \varphi_i(\mathbf{x})/\partial x_j \leq -\delta/(n-1)$, where $\mathbf{x} \equiv \mathbf{x}(t) = t\theta + (1-t)\theta'$.

Fix $0 \leq t \leq 1$. Since $\|\theta\|_\infty \leq K$ and $\|\theta'\|_\infty \leq K$, we also know that $\|\mathbf{x}\|_\infty \leq K$, so $-2K \leq x_i+x_j \leq 2K$ for all $1 \leq i,j \leq n$. Using the properties of $\mu$ and $\mu'$ from Lemma~\ref{Lem:PropertyMeanFunction}, we have:
\begin{equation*}
0 < \mu(2K) \leq \mu(x_i+x_j) \leq \mu(-2K)
\end{equation*}
and
\begin{equation*}
\mu'(0) \leq \mu'(x_i+x_j) \leq \mu'(2K) < 0.
\end{equation*}
Then from~\eqref{Eq:PartialVarphiMixed} and using the definition of $\delta$, it follows that:
\begin{equation*}
\frac{\partial \varphi_i(\mathbf{x})}{\partial x_j} \leq \frac{\mu'(2K)}{(n-1)(r-1)\mu(-2K)} \leq -\frac{\delta}{n-1}.
\end{equation*}
Furthermore, by Lemma~\ref{Lem:LowerBoundRatio} we have:
\begin{equation*}
\frac{\mu'(x_i+x_j)}{\mu(x_i+x_j)} \geq -r+1 + \frac{\exp(x_i+x_j)-1}{\exp(r(x_i+x_j))-1}
\geq -r+1 + \frac{\exp(2K)-1}{\exp(2rK)-1}.
\end{equation*}
So from~\eqref{Eq:PartialVarphi}, we also get:
\begin{equation*}
\begin{split}
\frac{\partial \varphi_i(\mathbf{x})}{\partial x_i}
&\geq 1 + \frac{1}{(r-1)} \frac{\sum_{j \neq i} (-r+1+\frac{\exp(2K)-1}{\exp(2rK)-1}) \mu(x_i+x_j)}{\sum_{j \neq i} \mu(x_i+x_j)}
= \frac{1}{(r-1)} \left( \frac{\exp(2K)-1}{\exp(2rK)-1} \right)
\geq \delta,
\end{split}
\end{equation*}
as required.
\end{proof}

We are now ready to prove Theorem~\ref{Thm:MLEAlgFiniteDisc}.

\begin{proof_of}{Theorem~\ref{Thm:MLEAlgFiniteDisc}}
By the mean-value theorem for vector-valued functions~\cite[p.~341]{Lang}, for any $\theta,\theta' \in \R^n$ we can write:
\begin{equation*}
\varphi(\theta) - \varphi(\theta') = J(\theta, \theta') (\theta-\theta'),
\end{equation*}
where $J(\theta,\theta')$ is the Jacobian matrix defined in~\eqref{Eq:JacobianMLEDef}. Since $\|J(\theta,\theta')\|_\infty = 1$ (Lemma~\ref{Lem:NormJInfinity}), this gives us:
\begin{equation}\label{Eq:FiniteDiscMVTBoundOneIter}
\|\varphi(\theta) - \varphi(\theta')\|_\infty \leq \|\theta-\theta'\|_\infty.
\end{equation}

First suppose there is a solution $\hat \theta$ to the system of equations~\eqref{Eq:MLEEqFiniteDisc}, so $\hat \theta$ is a fixed-point of $\varphi$. Then by setting $\theta = \theta^{(k)}$ and $\theta' = \hat \theta$ to the inequality above, we obtain:
\begin{equation}\label{Eq:VarphiContraction}
\|\theta^{(k+1)}-\hat\theta\|_\infty \leq \|\theta^{(k)}-\hat\theta\|_\infty.
\end{equation}
In particular, this shows that $\|\theta^{(k)}\|_\infty \leq K$ for all $k \in \N_0$, where $K := 2\|\hat\theta\|_\infty + \|\theta^{(0)}\|_\infty$. By Lemma~\ref{Lem:JInLnDelta}, this implies $J(\theta^{(k)},\hat \theta) \in \mathcal{L}_n(\delta)$ for all $k \in \N_0$, where $\delta$ is given by~\eqref{Eq:DeltaBoundJ}. Another application of the mean-value theorem gives us:
\begin{equation*}
\begin{split}
\theta^{(k+2)}-\hat \theta 
&= J(\theta^{(k+1)}, \hat\theta) \: J(\theta^{(k)}, \hat\theta) \: (\theta^{(k)} - \hat\theta),
\end{split}
\end{equation*}
so by Lemma~\ref{Lem:LnDelta}, we have:
\begin{equation*}
\begin{split}
\|\theta^{(k+2)}-\hat \theta\|_\infty
&\leq \|J(\theta^{(k+1)}, \hat\theta) \: J(\theta^{(k)}, \hat\theta)\|_\infty \: \|\theta^{(k)} - \hat\theta\|_\infty
\leq \left(1-\delta^2\right) \|\theta^{(k)} - \hat\theta\|_\infty.
\end{split}
\end{equation*}
Unrolling the recursive bound above and using~\eqref{Eq:VarphiContraction} gives us:
\begin{equation*}
\|\theta^{(k)}-\hat\theta\|_\infty \leq (1-\delta^2)^{\lfloor k/2 \rfloor} \|\theta^{(0)}-\hat\theta\|_\infty \leq (1-\delta^2)^{(k-1)/2} \|\theta^{(0)}-\hat\theta\|_\infty,
\end{equation*}
which proves~\eqref{Eq:MLERateOfConvFiniteDisc} with $\tau = \sqrt{1-\delta^2}$.

Now suppose the system of equations~\eqref{Eq:MLEEqFiniteDisc} does not have a solution, and suppose the contrary that the sequence $\{\theta^{(k)}\}$ does not have a divergent subsequence. This means $\{\theta^{(k)}\}$ is a bounded sequence, so there exists $K > 0$ such that $\|\theta^{(k)}\|_\infty \leq K$ for all $k \in \N_0$. Then by Lemma~\ref{Lem:JInLnDelta}, $J(\theta^{(k)},\theta^{(k+1)}) \in \mathcal{L}_n(\delta)$ for all $k \in \N_0$, where $\delta$ is given by~\eqref{Eq:DeltaBoundJ}. In particular, by the mean value theorem and Lemma~\ref{Lem:NormJInfinity}, we get for all $k \in \N_0$ that:
\begin{equation*}
\|\theta^{(k+3)}-\theta^{(k+2)}\|_\infty \leq (1-\delta^2) \|\theta^{(k+1)}-\theta^{(k)}\|_\infty.
\end{equation*}
This implies $\sum_{k=0}^\infty \|\theta^{(k+1)}-\theta^{(k)}\|_\infty < \infty$, which means $\{\theta^{(k)}\}$ is a Cauchy sequence. Thus, the sequence $\{\theta^{(k)}\}$ converges to a limit, say $\hat \theta$, as $k \to \infty$. This limit $\hat \theta$ is necessarily a fixed-point of $\varphi$, as well as a solution to the system of equations~\eqref{Eq:MLEEqFiniteDisc}, contradicting our assumption. Hence we conclude that $\{\theta^{(k)}\}$ must have a divergent subsequence.
\end{proof_of}

\subsection{Proofs for the infinite discrete weighted graphs}
\label{Sec:ProofInfiniteDisc}

In this section we prove the results presented in Section~\ref{Sec:InfiniteDisc}.

\subsubsection{Proof of Theorem~\ref{Thm:GraphicalInfiniteDisc}}

Without loss of generality we may assume $d_1 \geq d_2 \geq \dots \geq d_n$, so condition~\eqref{Eq:GraphicalInfiniteDisc} becomes $d_1 \leq \sum_{i=2}^n d_i$. The necessary part is easy: if $(d_1,\dots,d_n)$ is a degree sequence of a graph $G$ with edge weights $a_{ij} \in \N_0$, then $\sum_{i=1}^n d_i = 2\sum_{\{i,j\}} a_{ij}$ is even, and the total weight coming out of vertex $1$ is at most $\sum_{i=2}^n d_i$. For the converse direction, we proceed by induction on $s = \sum_{i=1}^n d_i$. The statement is clearly true for $s = 0$ and $s = 2$. Assume the statement is true for some even $s \in \N$, and suppose we are given $\mathbf{d} = (d_1,\dots,d_n) \in \N_0^n$ with $d_1 \geq \dots \geq d_n$, $\sum_{i=1}^n d_i = s+2$, and $d_1 \leq \sum_{i=2}^n d_i$. Without loss of generality we may assume $d_n \geq 1$, for otherwise we can proceed with only the nonzero elements of $\mathbf{d}$. Let $1 \leq t \leq n-1$ be the smallest index such that $d_t > d_{t+1}$, with $t = n-1$ if $d_1 = \dots = d_n$, and let $\mathbf{d}' = (d_1, \dots, d_{t-1}, d_t-1, d_{t+1}, \dots, d_n-1)$. We will show that $\mathbf{d}'$ is graphic. This will imply that $\mathbf{d}$ is graphic, because if $\mathbf{d}'$ is realized by the graph $G'$ with edge weights $a_{ij}'$, then $\mathbf{d}$ is realized by the graph $G$ with edge weights $a_{tn} = a_{tn}'+1$ and $a_{ij} = a_{ij}'$ otherwise.

Now for $\mathbf{d}' = (d_1', \dots, d_n')$ given above, we have $d_1' \geq \dots \geq d_n'$ and $\sum_{i=1}^n d_i' = \sum_{i=1}^n d_i-2 = s$ is even. So it suffices to show that $d_1' \leq \sum_{i=2}^n d_i'$, for then we can apply the induction hypothesis to conclude that $\mathbf{d}'$ is graphic. If $t = 1$, then $d_1' = d_1-1 \leq \sum_{i=2}^n d_i -1 = \sum_{i=2}^n d_i'$. If $t > 1$ then $d_1 = d_2$, so $d_1 < \sum_{i=2}^n d_i$ since $d_n \geq 1$. In particular, since $\sum_{i=1}^n d_i$ is even, $\sum_{i=2}^n d_i - d_1 = \sum_{i=1}^n d_i - 2d_1$ is also even, hence $\sum_{i=2}^n d_i - d_1 \geq 2$. Therefore, $d_1' = d_1 \leq \sum_{i=2}^n d_i-2 = \sum_{i=2}^n d_i'$. This finishes the proof of Theorem~\ref{Thm:GraphicalInfiniteDisc}.

\subsubsection{Proof of Lemma~\ref{Lem:ConvWInfiniteDisc}}

Clearly $\W \subseteq \W_1$, so $\overline{\conv(\W)} \subseteq \W_1$ since $\W_1$ is closed and convex, by Lemma~\ref{Lem:W-Convex}. Conversely, let $\Q$ denote the set of rational numbers. We will first show that $\W_1 \cap \Q^n \subseteq \conv(\W)$ and then proceed by a limit argument. Let $\mathbf{d} \in \W_1 \cap \Q^n$, so $\mathbf{d} = (d_1, \dots, d_n) \in \Q^n$ with $d_i \geq 0$ and $\max_{1 \leq i \leq n} d_i \leq \frac{1}{2} \sum_{i=1}^n d_i$. Choose $K \in \N$ large enough such that $Kd_i \in \N_0$ for all $i = 1,\dots,n$. Observe that $2K \mathbf{d} = (2Kd_1, \dots, 2Kd_n) \in \N_0^n$ has the property that $\sum_{i=1}^n 2Kd_i \in \N_0$ is even and $\max_{1 \leq i \leq n} 2Kd_i \leq \frac{1}{2} \sum_{i=1}^n 2Kd_i$, so $2K\mathbf{d} \in \W$ by definition. Since $0 = (0,\dots,0) \in \W$ as well, all elements along the segment joining $0$ and $2K\mathbf{d}$ lie in $\conv(\W)$, so in particular, $\mathbf{d} = (2K\mathbf{d})/(2K) \in \conv(\W)$. This shows that $\W_1 \cap \Q^n \subseteq \conv(\W)$, and hence $\overline{\W_1 \cap \Q^n} \subseteq \overline{\conv(\W)}$.

To finish the proof it remains to show that $\overline{\W_1 \cap \Q^n} = \W_1$. On the one hand, we have:
\begin{equation*}
\overline{\W_1 \cap \Q^n} \subseteq \overline{\W_1} \cap \overline{\Q^n} = \W_1 \cap \R_0^n = \W_1. 
\end{equation*}
For the other direction, given $\mathbf{d} \in \W_1$, choose $\mathbf{d}_1, \dots, \mathbf{d}_n \in W_1$ such that $\mathbf{d}, \mathbf{d}_1, \dots, \mathbf{d}_n$ are in general position, so that the convex hull $C$ of $\{\mathbf{d}, \mathbf{d}_1, \dots, \mathbf{d}_n\}$ is full dimensional. This can be done, for instance, by noting that the following $n+1$ points in $\W_1$ are in general position:
\begin{equation*}
\{0, \; \mathbf{e}_1 + \mathbf{e}_2, \; \mathbf{e}_1 + \mathbf{e}_3, \; \cdots, \; \mathbf{e}_1 + \mathbf{e}_n, \; \mathbf{e}_1 + \mathbf{e}_2 + \cdots + \mathbf{e}_n\},
\end{equation*}
where $\mathbf{e}_1, \dots, \mathbf{e}_n$ are the standard basis of $\R^n$. For each $m \in \N$ and $i = 1, \dots, n$, choose $\mathbf{d}_i^{(m)}$ on the line segment between $\mathbf{d}$ and $\mathbf{d}_i$ such that the convex hull $C_m$ of $\{\mathbf{d}, \mathbf{d}_1^{(m)}, \dots, \mathbf{d}_n^{(m)}\}$ is full dimensional and has diameter at most $1/m$. Since $C_m$ is full dimensional we can choose a rational point $\mathbf{r}_m \in C_m \subseteq C \subseteq \W_1$. Thus we have constructed a sequence of rational points $(\mathbf{r}_m)$ in $\W_1$ converging to $\mathbf{d}$, which shows that $\W_1 \subseteq \overline{\W_1 \cap \Q^n}$.

\subsubsection{Proof of Theorem~\ref{Thm:ConsistencyInfiniteDisc}}\label{Sec:ProofConsistencyInfiniteDisc}

We first address the issue of the existence of $\hat \theta$. Recall from the discussion in Section~\ref{Sec:InfiniteDisc} that the MLE $\hat \theta \in \Theta$ exists if and only if $\mathbf{d} \in \M^\circ$. Clearly $\mathbf{d} \in \W$ since $\mathbf{d}$ is the degree sequence of the sampled graph $G$, and $\W \subseteq \conv(\W) = \M$ from Proposition~\ref{Prop:MConvW}. Therefore, the MLE $\hat \theta$ does not exist if and only if $\mathbf{d} \in \partial \M = \M \setminus \M^\circ$, where the boundary $\partial \M$ is explicitly given by:
\begin{equation*}
\partial \M = \left\{ \mathbf{d}' \in \R_0^n \colon \min_{1 \leq i \leq n} d_i' = 0 \; \text{ or } \; \max_{1 \leq i \leq n} d_i' = \frac{1}{2} \sum_{i=1}^n d_i' \right\}.
\end{equation*}
Using union bound and the fact that the edge weights $A_{ij}$ are independent geometric random variables, it follows that:
\begin{equation*}
\begin{split}
\P(d_i = 0 \text{ for some } i)
\leq \sum_{i=1}^n \P(d_i = 0)
&= \sum_{i=1}^n \P(A_{ij} = 0 \text{ for all } j \neq i) \\
&= \sum_{i=1}^n \prod_{j \neq i} \left(1-\exp(-\theta_i-\theta_j)\right)
\leq n\left(1-\exp(-M)\right)^{n-1}.
\end{split}
\end{equation*}
Furthermore, again by union bound, we have:
\begin{equation*}
\P\left(\max_{1 \leq i \leq n} d_i = \frac{1}{2} \sum_{i=1}^n d_i \right)
= \P\left(d_i = \sum_{j \neq i} d_j \text{ for some } i \right)
\leq \sum_{i=1}^n \P\left(d_i = \sum_{j \neq i} d_j\right).
\end{equation*}
Note that we have $d_i = \sum_{j \neq i} d_j$ for some $i$ if and only if the edge weights $A_{jk} = 0$ for all $j,k \neq i$. This occurs with probability:
\begin{equation*}
\P\left(A_{jk} = 0 \text{ for } j,k \neq i \right)
= \prod_{\substack{j,k \neq i\\j \neq k}}\left(1-\exp(-\theta_j-\theta_k)\right)
\leq \left(1-\exp(-M)\right)^{\binom{n-1}{2}}.
\end{equation*}
Therefore, we have the bound:
\begin{equation*}
\begin{split}
\P(\mathbf{d} \in \partial \M)
&\leq \P(d_i = 0 \text{ for some } i) + \P\left(\max_{1 \leq i \leq n} d_i = \frac{1}{2} \sum_{i=1}^n d_i \right) \\
&\leq n\left(1-\exp(-M)\right)^{n-1} + n\left(1-\exp(-M)\right)^{\binom{n-1}{2}} \\
&\leq \frac{1}{n^{k-1}},
\end{split}
\end{equation*}
where the last inequality holds for sufficiently large $n$. This shows that for sufficiently large $n$, the MLE $\hat \theta$ exists with probability at least $1-1/n^{k-1}$.

We now turn to proving the consistency of $\hat \theta$. For the rest of this proof, assume that the MLE $\hat \theta \in \Theta$ exists, which occurs with probability at least $1-1/n^{k-1}$. The proof of the consistency of $\hat \theta$ follows the same outline as in the proof of Theorem~\ref{Thm:ConsistencyCont}. Let $\mathbf{d}^\ast = -\nabla Z(\theta)$ denote the expected degree sequence of the distribution $\P^\ast_\theta$, and recall that the MLE $\hat \theta$ satisfies $\mathbf{d} = -\nabla Z(\hat \theta)$. By the mean value theorem~\cite[p.~341]{Lang}, we can write:
\begin{equation}\label{Eq:Consistency-MVT-Disc}
\mathbf{d} - \mathbf{d}^\ast = \nabla Z(\theta) - \nabla Z(\hat \theta) = J(\theta - \hat\theta),
\end{equation}
where $J$ is the matrix obtained by integrating the Hessian of $Z$ between $\theta$ and $\hat \theta$:
\begin{equation*}
J = \int_0^1 \nabla^2 Z(t\theta + (1-t)\hat \theta) \: dt.
\end{equation*}

Let $0 \leq t \leq 1$, and note that at the point $\xi = t\theta + (1-t) \hat \theta$ the gradient $\nabla Z$ is given by:
\begin{equation*}
\big( \nabla Z(\xi) \big)_i = -\sum_{j \neq i} \frac{1}{\exp(\xi_i+\xi_j)-1}.
\end{equation*}
Thus, the Hessian $\nabla^2 Z$ is:
\begin{equation*}
\big( \nabla^2 Z(\xi) \big)_{ij} = \frac{\exp(\xi_i+\xi_j)}{(\exp(\xi_i+\xi_j)-1)^2}, \quad i \neq j,
\end{equation*}
\begin{equation*}
\big( \nabla^2 Z(\xi) \big)_{ii} = \sum_{j \neq i} \frac{\exp(\xi_i+\xi_j)}{(\exp(\xi_i+\xi_j)-1)^2} = \sum_{j \neq i} \big( \nabla^2 Z(\xi) \big)_{ij}.
\end{equation*}
Since $\theta, \hat \theta \in \Theta$ and we assume $\theta_i+\theta_j \leq M$, for $i \neq j$ we have:
\begin{equation*}
0 < \xi_i + \xi_j \leq \max\{\theta_i + \theta_j, \; \hat \theta_i + \hat \theta_j \} \leq \max\{M, \; 2\|\hat \theta\|_\infty\} \leq M + 2\|\hat \theta\|_\infty.
\end{equation*}
This means $J$ is a symmetric, diagonally dominant matrix with off-diagonal entries bounded below by $\exp(M + 2\|\hat\theta\|_\infty)/(\exp(M + 2\|\hat\theta\|_\infty)-1)^2$, being an average of such matrices. Then by Theorem~\ref{Thm:Main}, we have the bound:
\begin{equation*}
\|J^{-1}\|_\infty
\leq \frac{(3n-4)}{2(n-2)(n-1)} \: \frac{(\exp(M + 2\|\hat\theta\|_\infty)-1)^2}{\exp(M + 2\|\hat\theta\|_\infty)}
\leq \frac{2}{n} \: \frac{(\exp(M + 2\|\hat\theta\|_\infty)-1)^2}{\exp(M + 2\|\hat\theta\|_\infty)},
\end{equation*}
where the second inequality holds for $n \geq 7$. By inverting $J$ in~\eqref{Eq:Consistency-MVT-Disc} and applying the bound on $J^{-1}$ above, we obtain:
\begin{equation}\label{Eq:Consistency-R2}
\|\theta-\hat\theta\|_\infty \leq \|J^{-1}\|_\infty \: \|\mathbf{d} - \mathbf{d}^\ast\|_\infty \leq \frac{2}{n} \: \frac{(\exp(M + 2\|\hat\theta\|_\infty)-1)^2}{\exp(M + 2\|\hat\theta\|_\infty)} \: \|\mathbf{d} - \mathbf{d}^\ast\|_\infty.
\end{equation}

Let $A = (A_{ij})$ denote the edge weights of the sampled graph $G \sim \P^\ast_\theta$, so $d_i = \sum_{j \neq i} A_{ij}$ for $i = 1,\dots,n$. Since $\mathbf{d}^\ast$ is the expected degree sequence from the distribution $\P^\ast_\theta$, we also have $d_i^\ast = \sum_{j \neq i} 1/(\exp(\theta_i+\theta_j)-1)$. Recall that $A_{ij}$ is a geometric random variable with emission probability:
\begin{equation*}
q = 1-\exp(-\theta_i-\theta_j) \geq 1-\exp(-L),
\end{equation*}
so by Lemma~\ref{Lem:SubExp-Geo}, $A_{ij} - 1/(\exp(\theta_i + \theta_j)-1)$ is sub-exponential with parameter $-4/\log(1-q) \leq 4/L$. For each $i = 1,\dots,n$, the random variables $(A_{ij} - 1/(\exp(\theta_i+\theta_j)-1), j \neq i)$ are independent sub-exponential random variables, so we can apply the concentration inequality in Theorem~\ref{Thm:ConcIneqSubExp} with $\kappa = 4/L$ and:
\begin{equation*}
\epsilon = \left(\frac{16k \log n}{\gamma (n-1) L^2} \right)^{1/2}.
\end{equation*}
Assume $n$ is sufficiently large such that $\epsilon/\kappa = \sqrt{k \log n/\gamma (n-1)} \leq 1$. Then by Theorem~\ref{Thm:ConcIneqSubExp}, for each $i = 1,\dots,n$ we have:
\begin{equation*}
\begin{split}
\P\left(|d_i - d_i^\ast| \geq \sqrt{\frac{16 kn \log n}{\gamma L^2}} \right)
&\leq \P\left(|d_i - d_i^\ast| \geq \sqrt{\frac{16 k (n-1) \log n}{\gamma L^2}} \right) \\
&= \P\left(\Bigg|\frac{1}{n-1}\sum_{j \neq i} \left(A_{ij}-\frac{1}{\exp(\theta_i+\theta_j)-1}\right)\Bigg| \geq \sqrt{\frac{16 k \log n}{\gamma (n-1) L^2}} \right) \\
&\leq 2\exp\left(-\gamma \: (n-1) \cdot \frac{L^2}{16} \cdot \frac{16 k \log n}{\gamma (n-1) L^2} \right) \\
&= \frac{2}{n^k}.
\end{split}
\end{equation*}
The union bound then gives us:
\begin{equation*}
\P\left(\|\mathbf{d} - \mathbf{d}^\ast\|_\infty \geq \sqrt{\frac{16 kn \log n}{\gamma L^2}} \right)
\leq \sum_{i=1}^n \P\left(|d_i - d_i^\ast| \geq \sqrt{\frac{16 kn \log n}{\gamma L^2}} \right)
\leq \frac{2}{n^{k-1}}.
\end{equation*}

Assume now that $\|\mathbf{d} - \hat{\mathbf{d}}\|_\infty \leq \sqrt{16 kn \log n/(\gamma L^2)}$, which happens with probability at least $1-2/n^{k-1}$. From~\eqref{Eq:Consistency-R2} and using the triangle inequality, we get:
\begin{equation*}
\|\hat \theta\|_\infty
\leq \|\theta-\hat\theta\|_\infty + \|\theta\|_\infty
\leq \frac{8}{L} \: \sqrt{\frac{k\log n}{\gamma n}} \: \frac{(\exp(M + 2\|\hat \theta\|_\infty)-1)^2}{\exp(M + 2\|\hat \theta\|_\infty)} + M.
\end{equation*}
This means $\|\hat \theta\|_\infty$ satisfies the inequality $H_n(\|\hat \theta\|_\infty) \geq 0$, where $H_n(x)$ is the following function:
\begin{equation*}
H_n(x) = \frac{8}{L} \: \sqrt{\frac{k\log n}{\gamma n}} \: \frac{(\exp(M + 2x)-1)^2}{\exp(M + 2x)} - x + M.
\end{equation*}
One can easily verify that $H_n$ is a convex function, so $H_n$ assumes the value $0$ at most twice, and moreover, $H_n(x) \to \infty$ as $x \to \infty$. It is also easy to see that for all sufficiently large $n$, we have $H_n(2M) < 0$ and $H_n(\frac{1}{4} \log n) < 0$. Therefore, $H_n(\|\hat \theta\|_\infty) \geq 0$ implies either $\|\hat \theta\|_\infty < 2M$ or $\|\hat \theta\|_\infty > \frac{1}{4} \log n$. We claim that for sufficiently large $n$ we always have $\|\hat \theta\|_\infty < 2M$. Suppose the contrary that there are infinitely many $n$ for which $\|\hat \theta\|_\infty > \frac{1}{4} \log n$, and consider one such $n$. Since $\hat \theta_i + \hat \theta_j > 0$ for each $i \neq j$, there can be at most one index $i$ with $\hat \theta_i < 0$. We consider the following two cases.
\begin{enumerate}
  \item \textbf{Case 1:} suppose $\hat \theta_i \geq 0$ for all $i = 1,\dots,n$. Let $i^\ast$ be an index with $\hat \theta_{i^\ast} = \|\hat \theta\|_\infty > \frac{1}{4} \log n$. Then, since $\hat \theta_{i^\ast} + \hat \theta_j \geq \hat \theta_{i^\ast}$ for $j \neq i^\ast$, we have:
\begin{equation*}
\begin{split}
\frac{1}{\exp(M)-1} &\leq \frac{1}{n-1} \sum_{j \neq i^\ast} \frac{1}{\exp(\theta_{i^\ast} + \theta_j)-1} \\
&\leq \frac{1}{n-1} \left| \sum_{j \neq i^\ast} \frac{1}{\exp(\theta_{i^\ast}+\theta_j)-1} - \sum_{j \neq i^\ast} \frac{1}{\exp(\hat \theta_{i^\ast}+\hat \theta_j)-1} \right|
 + \frac{1}{n-1} \sum_{j \neq i^\ast} \frac{1}{\exp(\hat \theta_{i^\ast} + \hat \theta_j)-1} \\
&\leq \frac{1}{n-1} \| \mathbf{d} - \mathbf{d}^\ast \|_\infty + \frac{1}{\exp(\|\hat \theta\|_\infty)-1} \\
&\leq \frac{1}{n-1} \sqrt{\frac{16 kn \log n}{\gamma L^2}} + \frac{1}{n^{1/4}-1},
\end{split}
\end{equation*}
which cannot hold for sufficiently large $n$, as the last expression tends to $0$ as $n \to \infty$.

  \item \textbf{Case 2:} suppose $\hat \theta_i < 0$ for some $i = 1,\dots,n$, so $\hat \theta_j > 0$ for $j \neq i$. Without loss of generality assume $\hat \theta_1 < 0 < \hat \theta_2 \leq \cdots \leq \hat \theta_n$, so $\hat \theta_n = \|\hat \theta\|_\infty > \frac{1}{4} \log n$. Following the same chain of inequalities as in the previous case (with $i^\ast = n$), we obtain:
\begin{equation*}
\begin{split}
\frac{1}{\exp(M)-1} &\leq \frac{1}{n-1} \| \mathbf{d} - \mathbf{d}^\ast \|_\infty + \frac{1}{n-1} \left(\frac{1}{\exp(\hat \theta_n + \hat \theta_1)-1} + \sum_{j = 2}^{n-1} \frac{1}{\exp(\hat \theta_j + \hat \theta_n)-1} \right) \\
&\leq \frac{1}{n-1} \sqrt{\frac{16 kn \log n}{\gamma L^2}} + \frac{1}{(n-1)(\exp(\hat \theta_n + \hat \theta_1)-1)} + \frac{n-2}{(n-1)(\exp(\|\hat \theta\|_\infty)-1)} \\
&\leq \frac{1}{n-1} \sqrt{\frac{16 kn \log n}{\gamma L^2}} + \frac{1}{(n-1)(\exp(\hat \theta_n + \hat \theta_1)-1)} + \frac{1}{n^{1/4}-1}.
\end{split}
\end{equation*}
This implies:
\begin{equation*}
\begin{split}
\frac{1}{\exp(\hat \theta_1 + \hat \theta_n)-1}
&\geq (n-1)\left(\frac{1}{\exp(M)-1} - \frac{1}{n-1} \sqrt{\frac{16 kn \log n}{\gamma L^2}} - \frac{1}{n^{1/4}-1}\right)
\geq \frac{n}{2(\exp(M)-1)},
\end{split}
\end{equation*}
where the last inequality assumes $n$ is sufficiently large. Therefore, for $i = 2,\dots,n$:
\begin{equation*}
\frac{1}{\exp(\hat \theta_1 + \hat \theta_i)-1} \geq \frac{1}{\exp(\hat \theta_1 + \hat \theta_n)-1} \geq \frac{n}{2(\exp(M)-1)}.
\end{equation*}
However, this implies:
\begin{equation*}
\begin{split}
\sqrt{\frac{16 kn \log n}{\gamma L^2}} \geq \|\mathbf{d} - \mathbf{d}^\ast \|_\infty
\geq |d_1-d_1^\ast|
&\geq -\sum_{j=2}^n \frac{1}{\exp(\theta_1 + \theta_j)-1} +\sum_{j=2}^n \frac{1}{\exp(\hat\theta_1 + \hat\theta_n)-1} \\
&\geq - \frac{(n-1)}{\exp(L)-1} + \frac{n(n-1)}{2(\exp(M)-1)},
\end{split}
\end{equation*}
which cannot hold for sufficiently large $n$, as the right hand side in the last expression grows faster than the left hand side on the first line.
\end{enumerate}

The analysis above shows that we have $\|\hat \theta\|_\infty < 2M$ for all sufficiently large $n$. Plugging in this result to~\eqref{Eq:Consistency-R2} gives us:
\begin{equation*}
\begin{split}
\|\theta-\hat\theta\|_\infty
&\leq \frac{2}{n} \: \frac{(\exp(5M)-1)^2}{\exp(5M)} \: \sqrt{\frac{16 kn \log n}{\gamma L^2}}
\leq \frac{8 \: \exp(5M)}{L} \: \sqrt{\frac{k \log n}{\gamma n}}.
\end{split}
\end{equation*}
Finally, taking into account the issue of the existence of the MLE, we conclude that for sufficiently large $n$, with probability at least:
\begin{equation*}
\left(1-\frac{1}{n^{k-1}}\right)\left(1-\frac{2}{n^{k-1}}\right) \geq 1-\frac{3}{n^{k-1}},
\end{equation*}
the MLE $\hat \theta \in \Theta$ exists and satisfies:
\begin{equation*}
\|\theta-\hat\theta\|_\infty
\leq \frac{8 \: \exp(5M)}{L} \: \sqrt{\frac{k \log n}{\gamma n}},
\end{equation*}
as desired. This finishes the proof of Theorem~\ref{Thm:ConsistencyInfiniteDisc}.

\subsection{Proofs for the continuous weighted graphs}
\label{Sec:ProofCont}

In this section we present the proofs of the results presented in Section~\ref{Sec:Cont}.

\subsubsection{Proof of Theorem~\ref{Thm:GraphicalCont}}

Clearly if $(d_1, \dots, d_n) \in \R_0^n$ is a graphical sequence, then so is $(d_{\pi(1)}, \dots, d_{\pi(n)})$, for any permutation $\pi$ of $\{1,\dots,n\}$. Thus, without loss of generality we can assume $d_1 \geq d_2 \geq \cdots \geq d_n$, and in this case condition~\eqref{Eq:GraphicalCont} reduces to:
\begin{equation}\label{Eq:ErdosGallai-R2}
d_1 \leq \sum_{i=2}^n d_i.
\end{equation}

First suppose $(d_1, \dots, d_n) \in \R_0^n$ is graphic, so it is the degree sequence of a graph with adjacency matrix $\mathbf{a} = (a_{ij})$. Then condition~\eqref{Eq:ErdosGallai-R2} is satisfied since:
\begin{equation*}
d_1 = \sum_{i=2}^n a_{1i} \leq \sum_{i=2}^n \sum_{j \neq i} a_{ij} = \sum_{i=2}^n d_i.
\end{equation*}
For the converse direction, we first note the following easy properties of weighted graphical sequences:
\begin{enumerate}[(i)]
  \item\label{p:Cha} The sequence $(c, c, \dots, c) \in \R_0^n$ is graphic for any $c \in \R_0$, realized by the ``cycle graph'' with weights $a_{i,i+1} = c/2$ for $1 \leq i \leq n-1$, $a_{1n} = c/2$, and $a_{ij} = 0$ otherwise.

  \item\label{p:Eq} A sequence $\mathbf{d} = (d_1, \dots, d_n) \in \R_0^n$ satisfying~\eqref{Eq:ErdosGallai-R2} with an equality is graphic, realized by the ``star graph'' with weights $a_{1i} = d_i$ for $2 \leq i \leq n$ and $a_{ij} = 0$ otherwise.

  \item\label{p:Ext} If $\mathbf{d} = (d_1, \dots, d_n) \in \R_0^n$ is graphic, then so is $\overline{\mathbf{d}} = (d_1, \dots, d_n, 0, \dots, 0) \in \R_0^{n'}$ for any $n' \geq n$, realized by inserting $n'-n$ isolated vertices to the graph that realizes $\mathbf{d}$.
  
  \item\label{p:Sum} If $\mathbf{d}^{(1)}, \mathbf{d}^{(2)} \in \R_0^n$ are graphic, then so is $\mathbf{d}^{(1)} + \mathbf{d}^{(2)}$, realized by the graph whose edge weights are the sum of the edge weights of the graphs realizing $\mathbf{d}^{(1)}$ and $\mathbf{d}^{(2)}$.
\end{enumerate}

We now prove the converse direction by induction on $n$. For the base case $n = 3$, it is easy to verify that $(d_1,d_2,d_3)$ with $d_1 \geq d_2 \geq d_3 \geq 0$ and $d_1 \leq d_2 + d_3$ is the degree sequence of the graph $G$ with edge weights:
\begin{equation*}
a_{12} = \frac{1}{2}(d_1 + d_2 - d_3) \geq 0, \qquad
a_{13} = \frac{1}{2}(d_1 - d_2 + d_3) \geq 0, \qquad
a_{23} = \frac{1}{2}(-d_1 + d_2 + d_3) \geq 0.
\end{equation*}
Assume that the claim holds for $n-1$; we will prove it also holds for $n$. So suppose we have a sequence $\mathbf{d} = (d_1, \dots, d_n)$ with $d_1 \geq d_2 \geq \cdots \geq d_n \geq 0$ satisfying~\eqref{Eq:ErdosGallai-R2}, and set:
\begin{equation*}
K = \frac{1}{n-2} \left( \sum_{i=2}^n d_i - d_1 \right) \geq 0
\end{equation*}
If $K = 0$ then~\eqref{Eq:ErdosGallai-R2} is satisfied with an equality, and by property~\eqref{p:Eq} we know that $\mathbf{d}$ is graphic. Now assume $K > 0$. We consider two possibilities.
\begin{enumerate}
  \item Suppose $K \geq d_n$. Then we can write $\mathbf{d} = \mathbf{d}^{(1)} + \mathbf{d}^{(2)}$, where:
\begin{equation*}
\mathbf{d}^{(1)} = (d_1-d_n, \: d_2-d_n, \: \dots, \: d_{n-1}-d_n, \: 0) \in \R_0^n
\end{equation*}
and
\begin{equation*}
\mathbf{d}^{(2)} = (d_n, d_n, \dots, d_n) \in \R_0^n.
\end{equation*}
The assumption $K \geq d_n$ implies $d_1-d_n \leq \sum_{i=2}^{n-1} (d_i-d_n)$, so $(d_1-d_n, d_2-d_n, \dots, d_{n-1}-d_n) \in \R_0^{n-1}$ is a graphical sequence by induction hypothesis. Thus, $\mathbf{d}^{(1)}$ is also graphic by property~\eqref{p:Ext}. Furthermore, $\mathbf{d}^{(2)}$ is graphic by property~\eqref{p:Cha}, so $\mathbf{d} = \mathbf{d}^{(1)} + \mathbf{d}^{(2)}$ is also a graphical sequence by property~\eqref{p:Sum}.

  \item Suppose $K < d_n$. Then write $\mathbf{d} = \mathbf{d}^{(3)} + \mathbf{d}^{(4)}$, where:
\begin{equation*}
\mathbf{d}^{(3)} = (d_1-K, \: d_2-K, \: \dots, \: d_n-K) \in \R_0^n, and
\end{equation*}
\begin{equation*}
\mathbf{d}^{(4)} = (K, K, \dots, K) \in \R_0^n.
\end{equation*}
By construction, $\mathbf{d}^{(3)}$ satisfies $d_1-K = \sum_{i=2}^n (d_i-K)$, so $\mathbf{d}^{(3)}$ is a graphical sequence by property~\eqref{p:Eq}. Since $\mathbf{d}^{(4)}$ is also graphic by property~\eqref{p:Cha}, we conclude that $\mathbf{d} = \mathbf{d}^{(3)} + \mathbf{d}^{(4)}$ is graphic by property~\eqref{p:Sum}.
\end{enumerate}
This completes the induction step and finishes the proof of Theorem~\ref{Thm:GraphicalCont}.

\subsubsection{Proof of Lemma~\ref{Lem:W-Convex}}

We first prove that $\W$ is convex. Given $\mathbf{d} = (d_1, \dots, d_n)$ and $\mathbf{d}' = (d_1', \dots, d_n')$ in $\W$, and given $0 \leq t \leq 1$, we note that:
\begin{equation*}
\begin{split}
\max_{1 \leq i \leq n} \big( t d_i + (1-t) d_i' \big)
&\leq t \max_{1 \leq i \leq n} d_i + (1-t) \max_{1 \leq i \leq n} d_i' \\
&\leq \frac{1}{2} t \sum_{i=1}^n d_i + \frac{1}{2} (1-t) \sum_{i=1}^n d_i' \\
&= \frac{1}{2} \sum_{i=1}^n \big( t d_i + (1-t) d_i' \big),
\end{split}
\end{equation*}
which means $t\mathbf{d} + (1-t) \mathbf{d}' \in \W$.

Next, recall that we already have $\M \subseteq \conv(\W) = \W$ from Proposition~\ref{Prop:MConvW}, so to conclude $\M = \W$ it remains to show that $\W \subseteq \M$. Given $\mathbf{d} \in \W$, let $G$ be a graph that realizes $\mathbf{d}$ and let $\mathbf{w} = (w_{ij})$ be the edge weights of $G$, so that $d_i = \sum_{j \neq i} w_{ij}$ for all $i = 1,\dots,n$. Consider a distribution $\P$ on $\R_0^{\binom{n}{2}}$ that assigns each edge weight $A_{ij}$ to be an independent exponential random variable with mean parameter $w_{ij}$, so $\P$ has density:
\begin{equation*}
p(\mathbf{a}) = \prod_{\{i,j\}} \frac{1}{w_{ij}} \exp\left(-\frac{a_{ij}}{w_{ij}}\right),
\quad \mathbf{a} = (a_{ij}) \in \R_0^{\binom{n}{2}}.
\end{equation*}
Then by construction, we have $\E_\P[A_{ij}] = w_{ij}$ and:
\begin{equation*}
\E_\P[\deg_i(A)] = \sum_{j \neq i} \E_\P[A_{ij}] = \sum_{j \neq i} w_{ij} = d_i, \quad i = 1, \dots, n.
\end{equation*}
This shows that $\mathbf{d} \in \M$, as desired.

\subsubsection{Proof of Theorem~\ref{Thm:ConsistencyCont}}\label{Sec:ProofConsistencyCont}

We first prove that the MLE $\hat \theta$ exists almost surely. Recall from the discussion in Section~\ref{Sec:Cont} that $\hat \theta$ exists if and only if $\mathbf{d} \in \M^\circ$. Clearly $\mathbf{d} \in \W$ since $\mathbf{d}$ is the degree sequence of the sampled graph $G$. Since $\M = \W$ (Lemma~\ref{Lem:W-Convex}), we see that the MLE $\hat \theta$ does not exist if and only if $\mathbf{d} \in \partial \M = \M \setminus \M^\circ$, where:
\begin{equation*}
\partial \M = \left\{ \mathbf{d}' \in \R_0^n \colon \min_{1 \leq i \leq n} d_i' = 0 \; \text{ or } \; \max_{1 \leq i \leq n} d_i' = \frac{1}{2} \sum_{i=1}^n d_i' \right\}.
\end{equation*}
In particular, note that $\partial \M$ has Lebesgue measure $0$. Since the distribution $\P^\ast$ on the edge weights $A = (A_{ij})$ is continuous (being a product of exponential distributions) and $\mathbf{d}$ is a continuous function of $A$, we conclude that $\P^\ast(\mathbf{d} \in \partial \M) = 0$, as desired.

We now prove the consistency of $\hat \theta$. Recall that $\theta$ is the true parameter that we wish to estimate, and that the MLE $\hat \theta$ satisfies $-Z(\hat \theta) = \mathbf{d}$. Let $\mathbf{d}^\ast = -\nabla Z(\theta)$ denote the expected degree sequence of the maximum entropy distribution $\P^\ast_\theta$. By the mean value theorem for vector-valued functions~\cite[p.~341]{Lang}, we can write:
\begin{equation}\label{Eq:ConsistencyContMVT}
\mathbf{d} - \mathbf{d}^\ast = \nabla Z(\theta) - \nabla Z(\hat \theta) = J(\theta - \hat \theta).
\end{equation}
Here $J$ is a matrix obtained by integrating (element-wise) the Hessian $\nabla^2 Z$ of the log-partition function on intermediate points between $\theta$ and $\hat \theta$:
\begin{equation*}
J = \int_0^1 \nabla^2 Z(t \theta + (1-t) \hat \theta) \: dt.
\end{equation*}

Recalling that $-\nabla Z(\theta) = \E_\theta[\deg(A)]$, at any intermediate point $\xi \equiv \xi(t) = t \theta + (1-t) \hat \theta$, we have:
\begin{equation*}
\big(\nabla Z(\xi)\big)_i = -\sum_{j \neq i} \mu(\xi_i + \xi_j) = -\sum_{j \neq i} \frac{1}{\xi_i + \xi_j}.
\end{equation*} 
Therefore, the Hessian $\nabla^2 Z$ is given by:
\begin{equation*}
\big( \nabla^2 Z(\xi) \big)_{ij} = \frac{1}{(\xi_i + \xi_j)^2}, \quad i \neq j,
\end{equation*}
\begin{equation*}
\big( \nabla^2 Z(\xi) \big)_{ii} = \sum_{j \neq i} \frac{1}{(\xi_i+\xi_j)^2} = \sum_{j \neq i} \big( \nabla^2 Z(\xi) \big)_{ij}.
\end{equation*}
Since $\theta,\theta' \in \Theta$ and we assume $\theta_i+\theta_j \leq M$, it follows that for $i \neq j$:
\begin{equation*}
0 < \xi_i + \xi_j \leq \max\{\theta_i+\theta_j, \: \hat \theta_i + \hat \theta_j\} \leq \max\{M, 2\|\hat\theta\|_\infty\} \leq M + 2\|\hat \theta\|_\infty.
\end{equation*}
Therefore, the Hessian $\nabla^2 Z$ is a {\em diagonally balanced} matrix with off-diagonal entries bounded below by $1/(M + 2\|\hat \theta\|_\infty)^2$.
In particular, $J$ is also a symmetric, diagonally balanced matrix with off-diagonal entries bounded below by $1/(M + 2\|\hat \theta\|_\infty)^2$, being an average of such matrices. By Theorem~\ref{Thm:Main}, $J$ is invertible and its inverse satisfies the bound:
\begin{equation*}
\|J^{-1}\|_\infty \leq \frac{(M+2\|\hat\theta\|_\infty)^2(3n-4)}{2(n-1)(n-2)} \leq \frac{2}{n} \: (M + 2\|\hat \theta\|_\infty)^2,
\end{equation*}
where the last inequality holds for $n \geq 7$. Inverting $J$ in~\eqref{Eq:ConsistencyContMVT} and applying the bound on $\|J^{-1}\|_\infty$ gives:
\begin{equation}\label{Eq:Consistency-R1}
\|\theta-\hat\theta\|_\infty
\leq \|J^{-1}\|_\infty \: \|\mathbf{d} - \mathbf{d}^\ast\|_\infty
\leq \frac{2}{n} \: (M + 2\|\hat \theta\|_\infty)^2 \: \|\mathbf{d} - \mathbf{d}^\ast\|_\infty.
\end{equation}

Let $A = (A_{ij})$ denote the edge weights of the sampled graph $G \sim \P^\ast_\theta$, so $d_i = \sum_{j \neq i} A_{ij}$ for $i = 1,\dots,n$. Moreover, since $\mathbf{d}^\ast$ is the expected degree sequence from the distribution $\P^\ast_\theta$, we also have $d_i^\ast = \sum_{j \neq i} 1/(\theta_i+\theta_j)$. Recall that $A_{ij}$ is an exponential random variable with rate $\lambda = \theta_i + \theta_j \geq L$, so by Lemma~\ref{Lem:SubExp-Exp}, $A_{ij} - 1/(\theta_i + \theta_j)$ is sub-exponential with parameter $2/(\theta_i+\theta_j) \leq 2/L$. For each $i = 1,\dots,n$, the random variables $(A_{ij} - 1/(\theta_i+\theta_j), j \neq i)$ are independent sub-exponential random variables, so we can apply the concentration inequality in Theorem~\ref{Thm:ConcIneqSubExp} with $\kappa = 2/L$ and:
\begin{equation*}
\epsilon = \left(\frac{4k \log n}{\gamma (n-1) L^2} \right)^{1/2}.
\end{equation*}
Assume $n$ is sufficiently large such that $\epsilon/\kappa = \sqrt{k \log n / \gamma (n-1)} \leq 1$. Then by Theorem~\ref{Thm:ConcIneqSubExp}, for each $i = 1,\dots,n$ we have:
\begin{equation*}
\begin{split}
\P\left(|d_i - d_i^\ast| \geq \sqrt{\frac{4k n \log n}{\gamma L^2}} \right)
&\leq \P\left(|d_i - d_i^\ast| \geq \sqrt{\frac{4 k (n-1) \log n}{\gamma L^2}} \right) \\
&= \P\left(\Bigg|\frac{1}{n-1}\sum_{j \neq i} \left(A_{ij}-\frac{1}{\theta_i+\theta_j}\right)\Bigg| \geq \sqrt{\frac{4k \log n}{\gamma (n-1) L^2}} \right) \\
&\leq 2\exp\left(-\gamma \: (n-1) \cdot \frac{L^2}{4} \cdot \frac{4k \log n}{\gamma (n-1) L^2}\right) \\
&= \frac{2}{n^k}.
\end{split}
\end{equation*}
By the union bound, it follows that:
\begin{equation*}
\begin{split}
\P\Bigg(\|\mathbf{d} - \mathbf{d}^\ast\|_\infty \geq \sqrt{\frac{4k n \log n}{\gamma L^2}} \; \Bigg)
&\leq \sum_{i=1}^n \P\left(|d_i - d_i^\ast| \geq \sqrt{\frac{4k n \log n}{\gamma L^2}}\right)
\leq \frac{2}{n^{k-1}}.
\end{split}
\end{equation*}

Assume for the rest of this proof that $\|\mathbf{d} - \mathbf{d}^\ast\|_\infty \leq \sqrt{4kn \log n/(\gamma L^2)}$, which happens with probability at least $1-2/n^{k-1}$. From~\eqref{Eq:Consistency-R1} and using the triangle inequality, we get:
\begin{equation*}
\|\hat \theta\|_\infty
\leq \|\theta-\hat\theta\|_\infty + \|\theta\|_\infty
\leq \frac{4}{L} \: \sqrt{\frac{k \log n}{\gamma n}} \: (M + 2\|\hat \theta\|_\infty)^2 + M.
\end{equation*}
What we have shown is that for sufficiently large $n$, $\|\hat \theta\|_\infty$ satisfies the inequality $G_n(\|\hat \theta\|_\infty) \geq 0$, where $G_n(x)$ is the quadratic function:
\begin{equation*}
G_n(x) = \frac{4}{L} \: \sqrt{\frac{k \log n}{\gamma n}} \: (M + 2x)^2 - x + M.
\end{equation*}
It is easy to see that for sufficiently large $n$ we have $G_n(2M) < 0$ and $G_n(\log n) < 0$. Thus, $G_n(\|\hat \theta\|_\infty) \geq 0$ means either $\|\hat \theta\|_\infty < 2M$ or $\|\hat \theta\|_\infty > \log n$. We claim that for sufficiently large $n$ we always have $\|\hat \theta\|_\infty < 2M$. Suppose the contrary that there are infinitely many $n$ for which $\|\hat \theta\|_\infty > \log n$, and consider one such $n$. Since $\hat \theta \in \Theta$ we know that $\hat \theta_i + \hat \theta_j > 0$ for each $i \neq j$, so there can be at most one index $i$ with $\hat \theta_i < 0$. We consider the following two cases.
\begin{enumerate}
  \item \textbf{Case 1:} suppose $\hat \theta_i \geq 0$ for all $i = 1,\dots,n$. Let $i^\ast$ be an index with $\hat \theta_{i^\ast} = \|\hat \theta\|_\infty > \log n$. Then, using the fact that $\hat \theta$ satisfies the system of equations~\eqref{Eq:MLEEqCont} and $\hat \theta_{i^\ast} + \hat \theta_j \geq \hat \theta_{i^\ast}$ for $j \neq i^\ast$, we see that:
\begin{equation*}
\begin{split}
\frac{1}{M} &\leq \frac{1}{n-1} \sum_{j \neq i^\ast} \frac{1}{\theta_{i^\ast} + \theta_j} \\
&\leq \frac{1}{n-1} \left| \sum_{j \neq i^\ast} \frac{1}{\theta_{i^\ast}+\theta_j} - \sum_{j \neq i^\ast} \frac{1}{\hat \theta_{i^\ast}+\hat \theta_j} \right| + \frac{1}{n-1} \sum_{j \neq i^\ast} \frac{1}{\hat \theta_{i^\ast} + \hat \theta_j} \\
&= \frac{1}{n-1} \left| d_i^\ast - d_i \right| + \frac{1}{n-1} \sum_{j \neq i^\ast} \frac{1}{\hat \theta_{i^\ast} + \hat \theta_j} \\
&\leq \frac{1}{n-1} \| \mathbf{d}^\ast - \mathbf{d} \|_\infty + \frac{1}{\|\hat \theta\|_\infty} \\
&\leq \frac{1}{n-1} \: \sqrt{\frac{4k n \log n}{\gamma L^2}} + \frac{1}{\log n},
\end{split}
\end{equation*}
which cannot hold for sufficiently large $n$, as the last expression tends to $0$ as $n \to \infty$.

  \item \textbf{Case 2:} suppose $\hat \theta_i < 0$ for some $i = 1,\dots,n$, so $\hat \theta_j > 0$ for $j \neq i$ since $\hat \theta \in \Theta$. Without loss of generality assume $\hat \theta_1 < 0 < \hat \theta_2 \leq \cdots \leq \hat \theta_n$, so $\hat \theta_n = \|\hat\theta\|_\infty > \log n$. Following the same chain of inequalities as in the previous case (with $i^\ast = n$), we obtain:
\begin{equation*}
\begin{split}
\frac{1}{M} &\leq \frac{1}{n-1} \| \mathbf{d}^\ast - \mathbf{d} \|_\infty + \frac{1}{n-1} \left(\frac{1}{\hat \theta_n + \hat \theta_1} + \sum_{j = 2}^{n-1} \frac{1}{\hat \theta_j + \hat \theta_n} \right) \\
&\leq \frac{1}{n-1} \: \sqrt{\frac{4k n \log n}{\gamma L^2}} + \frac{1}{(n-1)(\hat \theta_n + \hat \theta_1)} + \frac{n-2}{(n-1)\|\hat \theta\|_\infty} \\
&\leq \frac{1}{n-1} \: \sqrt{\frac{4k n \log n}{\gamma L^2}} + \frac{1}{(n-1)(\hat \theta_n + \hat \theta_1)} + \frac{1}{\log n}.
\end{split}
\end{equation*}
So for sufficiently large $n$:
\begin{equation*}
\frac{1}{\hat \theta_1 + \hat \theta_n} \geq (n-1)\left(\frac{1}{M} - \frac{1}{n-1} \: \sqrt{\frac{4k n \log n}{\gamma L^2}} - \frac{1}{\log n}\right) \geq \frac{n}{2M},
\end{equation*}
and thus $\hat \theta_1 + \hat \theta_i \leq \hat \theta_1 + \hat \theta_n \leq 2M/n$ for each $i = 2,\dots,n$. However, then:
\begin{equation*}
\begin{split}
\sqrt{\frac{4k n \log n}{\gamma L^2}} &\geq \|\mathbf{d}^\ast - \mathbf{d} \|_\infty
\geq |d_1^\ast - d_1|
\geq -\sum_{j=2}^n \frac{1}{\theta_1 + \theta_j} +\sum_{j=2}^n \frac{1}{\hat\theta_1 + \hat\theta_j}
\geq - \frac{(n-1)}{L} + \frac{n(n-1)}{2M},
\end{split}
\end{equation*}
which cannot hold for sufficiently large $n$, as the right hand side of the last expression tends to $\infty$ faster than the left hand side.
\end{enumerate}
The analysis above shows that $\|\hat \theta\|_\infty < 2M$ for all sufficiently large $n$. Plugging in this result to~\eqref{Eq:Consistency-R1}, we conclude that for sufficiently large $n$, with probability at least $1-2n^{-(k-1)}$, we obtain the desired bound:
\begin{equation*}
\|\theta-\hat\theta\|_\infty \leq \frac{2}{n} \: (5M)^2 \: \sqrt{\frac{4k n \log n}{\gamma L^2}} = \frac{100M^2}{L} \sqrt{\frac{k \log n}{\gamma n}}.
\end{equation*}

\section{Discussion and future work}
\label{Sec:Discussion}

In this paper, we studied maximum entropy distributions on undirected graphs with a given expected degree sequence; in particular, we focused on three weight classes: finite discrete graphs, infinite discrete graphs, and continuous graphs. The corresponding models are characterized by independent edge weights parameterized by a vector of vertex potentials. We examined the problem of finding the maximum likelihood estimate of the original ensemble parameters, and we proved its remarkable consistency from a single graph sample.

In the case of finite discrete weighted graphs, we also give an efficient fixed-point algorithm for finding the MLE with a geometric rate of convergence. On the other hand, computing the MLE for infinite discrete or continuous weighted graphs can be performed via standard gradient-based methods, and the bounds we have proved on the inverse Hessian of the log-partition function can provide a rate of convergence for these methods. However, it would be interesting to develop fast algorithms (e.g., using block coordinate methods \cite{xu2013block}) for computing the MLE, similar to the case of finite discrete graphs.  From the standpoint of neuroscience theory, moreover, is there an approach that provides for neurally plausible algorithms?  We remark that one strategy could be to generalize Theorem~\ref{Thm:MLEAlgFiniteDisc} to settings such as the unbounded models studied here.

\begin{problem}
Develop fast (biologically plausible) algorithms to solve for maximum entropy  parameters.
\end{problem}

Recently, Bernstein's inequality \cite{bernstein1924} was applied in \cite{hillartran} to demonstrate robust exponential storage of certain graphs (cliques) in the Hopfield discrete recurrent neural network model.  
As part of a general discussion on using large deviation theory to understand efficiency of neural systems, we offer the following problem.

\begin{figure}[t!]
	\begin{center}
\includegraphics[width=1.5 in]{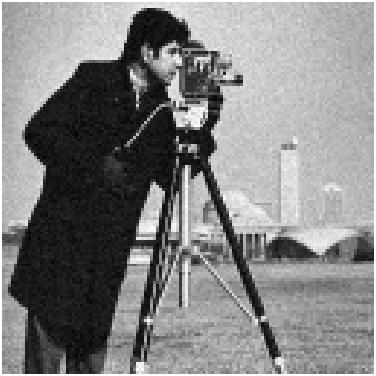}
\includegraphics[width=1.7 in]{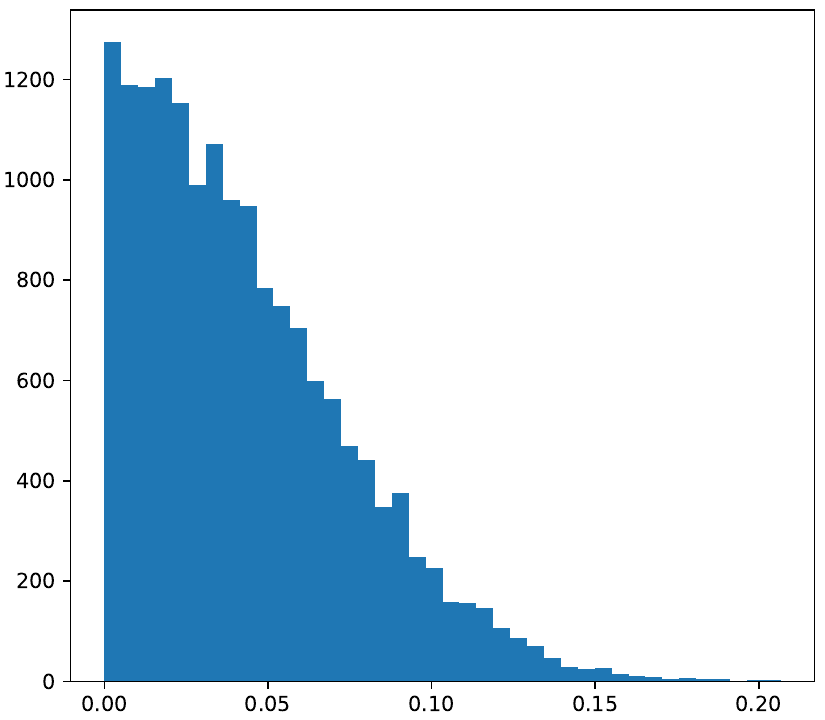} 
\includegraphics[width=1.5 in]{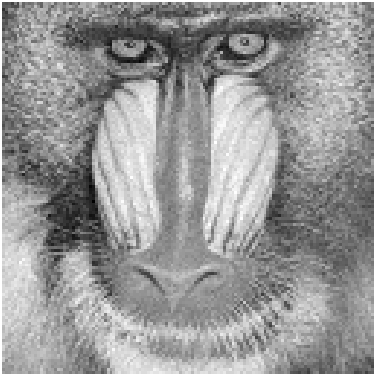}
\includegraphics[width=1.7 in]{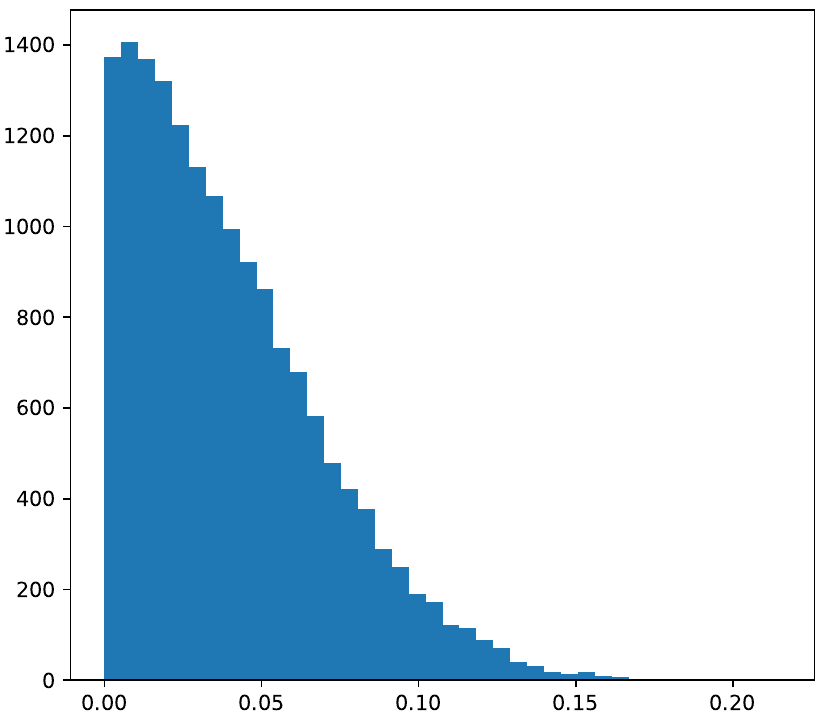} 

\caption{\small{\textbf{Sparse graph distributions for sensor quantization}. Reconstructions using Theorem~\ref{Thm:MLEAlgFiniteDisc} as in third column of Figure \ref{fig_imgs_app_1} except that only $10\%$ of edges are allowed to be active in a graph sample. To the right of each reconstruction is a histogram of absolute differences between original (first column of Figure \ref{fig_imgs_app_1}) and reconstruction.}}
\label{fig_imgs_app_sparse}
\end{center}
\end{figure}


\begin{problem}
Develop more applications of large deviation theory to (biologically plausible) computation.
\end{problem}

Given a fixed distortion tolerable by a system, it is known that there is an optimal (in terms of rate or coding cost) stochastic representation of the input with this expected error. Somewhat remarkably, such optimal codings can often necessitate a discrete alphabet, according to classical work \cite{smith1971information, fix1978rate, Rose94} (see \cite{jung2015discrete} for recent applications to behavioral economics and \cite{hillarmarzen} for bio-inspired image compression).
This might explain some of the discrete aspects of brains, but sometimes coding cost is not the only constraint in a biological system, and thus does not obviate continuous sampling strategies such as spike timing.  Even in this continuous setting, an unresolved controversy in neuroscience is whether information is contained in the precise timings of these spikes or primarily in their firing rates.  
This motivates the following challenge.
\begin{problem}
Develop a computational theory of optimal codings of continuous distributions.
\end{problem}

Another interesting research direction is to explore the theory of maximum entropy distributions when additional restrictions on the underlying graph are imposed. For instance, one can start with an arbitrary graph  $G_0$ on $n$ vertices, such as a lattice graph or a sparse graph, and consider the maximum entropy distributions on the subgraphs $G$ of $G_0$. By choosing different types of underlying structure $G_0$, we can incorporate additional prior information from the specific applications in consideration.  

\begin{question}
How does the mathematics change when the graph (hypergraph) is restricted to a fixed $G_0$?
\end{question}

For generalizations to a hypergraph version of the $\beta$-model, see \cite{stasi2014beta}. Also, some recent work on sparse random graphs in a physics context appears in \cite{van2017sparse} (and 
see \cite{mukherjee2018detection} for applications to real world networks). The image quantization example from the  introduction also works with sparsity imposed on the underlying graph.  See Figure 
\ref{fig_imgs_app_sparse}, which shows reconstructions and errors when only $10\%$ of edges are allowed to be active.


Given the connections briefly touched upon here, we hope that much more interesting algebraic geometry emerges from this circle of ideas.

\begin{question}
What is the algebraic geometry for the first two sets of equations from the introduction?
\end{question}

Finally, and more speculatively, connections to the theory of graph limits  
\cite{lovasz2006limits}, which were an inspiration for the work of \cite{Chatterjee}, should also be explored, especially in the unbounded edge weight cases.

\bibliographystyle{plain}
\bibliography{maxentgraphs}

\end{document}